\documentclass{article}

\usepackage{amssymb,amsthm,units,stmaryrd}

\usepackage{tikz}
\usepackage{pgf}
\usetikzlibrary{automata,arrows,shapes,snakes,topaths,trees,backgrounds,shadows}

\newtheorem{definition}{Definition}
\newtheorem{lemma}{Lemma}

\newtheorem{fact}{Fact}
\newtheorem{corollary}{Corollary}
\newtheorem{proposition}{Proposition}
\newtheorem{theorem}{Theorem}

\usepackage{graphicx}
\usepackage{amsmath}
\usepackage{amssymb}
\usepackage{stmaryrd}

\def\logic{{\sf Ev}}
\def\pl{{\langle\peq\rangle}}
\def\nl{{[\peq]}}
\def\nc{{[E]}}
\def\ps{{\langle E\rangle}}
\def\<{\langle}
\def\>{\rangle}

\def\peq{\preccurlyeq}
\def\seq{\succcurlyeq}

\def\L{\mathcal{L}}
 \def\M{\mathcal{M}}

\def\At{\mathsf{At}}
 \newcommand{\mo}[1]{\mathcal{{#1}}}      
 \newcommand{\pmo}[1]{\mathcal{{#1}}}      

 \newcommand{\truth}[2]{[\![{#1}]\!]_{#2}}      
\def\pow{\wp}

\def\lb{\llbracket}
\def\rb{\rrbracket}

\def\nB{{[B]}}
\def\nA{{[A]}}
\def\pB{{\langle B\rangle}}
\def\pA{{\langle A\rangle}}

\def\peq{\preccurlyeq}
\def\seq{\succcurlyeq}

\def\cbra{\left \{}
\def\cket{\right \}}

\DeclareSymbolFont{AMSb}{U}{msb}{m}{n}
 \DeclareMathSymbol{\Z}{\mathbin}{AMSb}{"5A}
\DeclareMathSymbol{\R}{\mathbin}{AMSb}{"52}
\DeclareMathSymbol{\Q}{\mathbin}{AMSb}{"51}
\DeclareMathSymbol{\I}{\mathbin}{AMSb}{"49}

\def\X{\mathcal{X}}
\def\phi{\varphi}

\def\E{\mathcal{E}}

\begin{document}



\title{Evidence and plausibility in neighborhood structures}


\author{Johan van Benthem\\
University of Amsterdam\\\\
David Fern\'{a}ndez-Duque\\
Universidad de Sevilla\\\\
Eric Pacuit\\
University of Maryland}
\maketitle

\begin{abstract}
The intuitive notion of evidence has both semantic and syntactic features. In this paper, we develop an {\em evidence logic} for epistemic agents faced with possibly contradictory evidence from different sources. The logic is based on a neighborhood semantics, where a neighborhood $N$ indicates that the agent has reason to believe that the true state of the world lies in $N$. Further notions of relative plausibility between worlds and beliefs based on the latter ordering are then defined in terms of this evidence structure, yielding our intended models for evidence-based beliefs. In addition, we also consider a second more general flavor, where belief and plausibility are modeled using additional primitive relations, and we prove a representation theorem showing that each such general model is a $p$-morphic image of an intended one. This semantics invites a number of natural special cases, depending on how uniform we make the
evidence sets, and how coherent their total structure. We give a structural study of the resulting `uniform' and `flat'  models. Our main result are sound and complete axiomatizations for the logics of all four major model classes with respect to the modal language of evidence, belief and safe belief. We conclude with an outlook toward logics for the dynamics of changing evidence, and the resulting language extensions and connections with logics of plausibility change.
\end{abstract}



\section{Introduction}

It has become standard practice in Artificial Intelligence and Game Theory to use possible-worlds models to describe the knowledge and beliefs of a group of agents.  In such models, the agents' knowledge is based on what is true throughout the set of epistemically accessible   worlds, the current information range.    Following a similar pattern, the agents' beliefs are based on what is true in the set of most ``plausible" worlds.    Now, it is often implicitly assumed that the  agents arrived at these structures through some process of investigation, but these details are no longer present in the models. 

 In a number of areas, ranging from epistemology to computer science and decision theory, the need has been recognized for models that keep track of the ``reasons'', or the {\em evidence} for beliefs and other informational attitudes (cf. \cite{list-dietrich,horty,sep-logic-justification,su1}).     Encoding evidence as the current range of worlds the agent considers possible ignores how the agent arrived at this epistemic state. This also ignores the fine-structure of evidence that allows us to consider or modify just parts of it. One extreme for recording this additional structure are models with complete syntactic details of what the agent has learned so far (including the precise formulation and sources for each piece of evidence) (cf. \cite{johan-fer-synthese}). In this paper, we will explore a middle ground in between ranges of possible worlds and syntactic fine-structure, viz. neighborhood models, where the available evidence is recorded as a family of sets of worlds. 
Our models are not unlike some earlier proposals in the study of conditionals and belief revision (cf.  \cite{kratzer,veltman,gardenfors,rott}), but we quickly strike out in other directions, and provide a more in-depth logical treatment.

This paper is a continuation of our earlier work in \cite{vBP-evidence,vBFP-aiml} on a new evidence interpretation of  neighborhood models.\footnote{Neighborhood models have been used to provide a semantics for both normal and non-normal modal logics.  See \cite{Segerberg} for an early discussion of neighborhood models and their logics, and  \cite{hansen-thesis,pacuit-nbhd,helle-clemens-ep} for modern motivations and mathematical details.} In a neighborhood model, each state is assigned a collection of subsets of the set of states.  We view  these  different collections of subsets    as    the evidence that the agent has acquired -- allowing the agent to have different evidence at different states.    Given such an explicit description of the evidence that the agent has acquired, one can explore different notions of beliefs and related cognitive attitudes over neighborhood models. The logical systems that arise on natural model classes of this sort with modalities for evidence and belief have been axiomatized completely in \cite{vBFP-aiml}. In this paper, we go one step further, and add some further crucial structure to the neighborhood structures.

 In general,  there are two ways we can enrich neighborhood structures with descriptions of the agent's beliefs.  The first approach is to   add  new accessibility relations corresponding to each epistemic or doxastic attitude in  the neighborhood structure.   We then impose constraints on these new relations to ensure that they are ``appropriately grounded'' on the available evidence.   The second approach is to define the agent's beliefs and related cognitive attitudes directly using no more than the given evidence structure in the  neighborhood structures. The latter ``intended models'' may be considered as a special case of the former ``general models''. This paper offers a careful study of these two approaches.

 Our second new contribution is to  elaborate on the relationship between our neighborhood models and another    general framework for belief change in the modal tradition.   Originally used as a semantics for conditionals   \cite[cf.][]{lewis-conditionals},  {\em plausibility models} are wide-spread in modal logics of belief  \cite{vb-logbelrev,vanbenthem-newbook,baltag-smets,girard}.  The main idea is to endow epistemic ranges with  an ordering $w\peq v$ of relative {\em plausibility} on worlds (usually uniform across epistemic equivalence classes): ``(according to the agent) world $v$ is at least as plausible as $w$".\footnote{In conditional semantics, such plausibility or `similarity' orders are typically world-dependent.}   Plausibility orders are typically assumed to be reflexive and transitive, and often even {\em connected}, making every two worlds comparable. Connections between evidence structure and plausibility order will occur throughout this paper, and their reflection in logical axioms will be determined. 
 
In all, we shall consider four variants of the logic of evidence-based belief, which depend on some fundamental assumptions one may make about evidence models, in terms of uniformity of evidence across worlds, and total coherence of maximally consistent sets of evidence. For each of the resulting logics, we prove two main results. The first is a characterization theorem for general evidence models as being {\em p-morphic images} of intended models. The second result determines a complete deductive calculus for each logic. Here our representation using extended evidence models is crucial, since it permits us to employ familiar techniques from modal logic.

\medskip

Anyone familiar with Sergei Artemov's work will have seen a similarity by now. It is very natural to attach to every believed proposition a ``justification" for that proposition.  This idea was first studied in Artemov's seminal paper \cite{artemov} and applied to epistemic logic in   \cite{artemov:2005} (see \cite{sep-logic-justification} for an overview and pointers to the relevant literature). In particular, $t\!:\!\phi$ is intended to mean that the agent believes $\phi$ and that $t$ is the justification for this belief. Here  $t$ is a proof term, and sophisticated logical systems have been developed that extend traditional ``provability logic'' to reason explicitly about the logical structure of these justifications.\footnote{By now, these systems have been applied to a variety of issues in epistemology and game theory, and other fields: see \cite{sep-logic-justification} for the relevant references.} Moreover, an appealing semantics exists for these systems (cf. \cite{fitting-jl}) that combines ideas from possible worlds semantics and the syntax of proof terms. The approach in the current paper is more coarse-grained, since we do not give ourselves proof or evidence terms that can be manipulated in our deductive systems. Nevertheless, we hope to show that even at our chosen level of modeling evidence, an amazing amount of fine-structure exists. A deeper comparison between our approach and justification logic seems a well-worth effort, but it is one step beyond the horizon of this paper.\footnote{For the moment, we can only refer to \cite{baltag-etal:12:jbc-wollic} for a system that merges features of justification logic with plausibility models and syntax dynamics in the style of ``dynamic-epistemic logic'' \cite{vanbenthem-newbook}, an enterprise close to the spirit of the present paper.}

\medskip

This paper is organized as follows.  The next section (Section \ref{basiclogic}) introduces the logical systems that we will study.  Each of these logical systems is defined using sublanguages of a single modal language, which includes  both  a  non-normal modality (the ``evidence-for'' modality) and normal modalities (two belief modalities and a universal modality).  One important theme in this paper  is that there are two different types of models that can be used as a semantics for this language.     The first type of model is a   neighborhood structure where the relations used to interpret the normal modalities are {\em derived} from the neighborhood function.    The second type of model  extends neighborhood models with relations, one for each belief operator.      Our first technical contribution   is to clarify the  relationship between these two types of models  (Sections \ref{derived} and \ref{representation}).   More generally, we compare these models   with the more standard   {\em plausibility models} often used to represent a rational agent's belief in Section \ref{belplaus}.    Sections \ref{syntactic-properties} and \ref{completeness} contain our second main contribution of this paper: proofs that our axiom systems are complete with respect to their intended class of models.  Section \ref{lang-ext} returns to  the original motivation for studying evidence logic from \cite{vBP-evidence}: developing dynamic logics of evidence management.     We offer some concluding remarks in Section \ref{conclusion}.

 \section{Logics of evidence and belief}\label{basiclogic}

    We start by presenting   the logical system that we will study in this paper.  Given a set $W$ of possible worlds or states, one of which represents the ``actual" situation, an agent gathers evidence about this situation from a variety of sources.    To simplify things, we assume these sources provide {\em binary} evidence, i.e., subsets of $W$ which (may) contain the actual world.           The following modal language can be used to describe what the agent believes given her available evidence (cf.  \cite{vBP-evidence}).  
  
  \begin{definition}\label{ev-lang} Let $\At$ be a fixed set of atomic propositions.   Let $\L=\L_{ABE\peq}$ be the smallest set of formulas generated
by the following grammar $$p\ |\ \neg\phi\ |\ \phi\wedge\psi\  |\ \nB\phi\ |\ \nc\phi \  |\ \nA\phi \  |\ \nl\phi$$
where $p\in\At$.  The propositional connectives ($\wedge, \rightarrow,\leftrightarrow$) are defined as usual and the duals\footnote{In other words, $\pB=\neg \nB\neg$, and similarly for other operators.} of $\nc, \nB,\nl$ and $\nA$ are $\ps, \nB$, $\pl$ and  $\pA$, respectively.

A {\em signature} is a subset of $\{A,B,E,\peq\}$, written usually as a string rather than a set. Sublanguages $\L'\subseteq\L$ will be denoted by $\L_{\sigma}$, where $\sigma$ is a signature, and $\L'$ contains only formulas with modalities in the signature.
\end{definition}

The intended interpretation of $\nc\phi$ is ``the agent has evidence for $\phi$" and   $\nB\phi$ says that ``the agents believes that $\phi$ is true."  The modality $\nl\phi$ refers to the plausibility order to be introduced below, and it has become standard to read it as ``the agent {\em safely believes} that $\phi$ is true", though we will really explore this interpretation only in Section \ref{belplaus}.      We also include the universal modality ($\nA\phi$: ``$\phi$ is true in all states'') for technical convenience.\footnote{A natural interpretation of $\nA\phi$ in the single-agent context of this paper is as ``the agent knows that $\phi$".  }  
   
Since we do not assume that the sources of the evidence are jointly consistent (or even that a single source is guaranteed to be consistent and provide {\em all} the available evidence), the ``evidence for" operator ($\nc\phi$) is not a normal modal operator.   That is, the  agent may have evidence for $\phi$ and evidence for $\psi$ ($\nc\phi\wedge\nc\psi)$ without having  evidence for their conjunction ($\neg\nc(\phi\wedge\psi)$).  Of course,   the two belief and  universal operators are normal modal operators.   So, the logical systems we study in this paper  {\em  combine} a non-normal modal logic with a normal one.

      \subsection{Models for $\L$ and its fragments}\label{ev-models}
      
     Our background assumption is that the   agent   uses the evidence she has gathered at each state to form her beliefs. There are two ways to make this precise.    The first is to ground the agent's beliefs (defined as relations on the set of states) on the evidence available at each state via a number of technical assumptions.   The second way is to define the agent's beliefs directly in terms of the evidence available at each state.    These different options lead us to consider several different signatures.   This motivates the following general definition of a {\em model}:

  \begin{definition}\label{ev-model} Let $\sigma$ be a signature containing $E$ and not containing\footnote{The universal modality will {\em always} be interpreted according to its intended meaning and need not be represented explicitly in the models.  In addition, we consider evidence the most fundamental notion in our setting and assume it always in our signature.} $A$. A {\bf $\sigma$-structure} is a tuple $\mo{M}=\langle W,  \langle I(m)\rangle_{m\in\sigma}, V\rangle$, where $W$ is a non-empty set of worlds and
\begin{itemize}
\item if $E\in\sigma$ then $I(E)\subseteq W\times \pow(W)$ is an {\em evidence relation,}
\item if $\mathrel\peq\in\sigma$ then $I(\peq)$ is a plausibility order on $W$: $I(\peq)$ is reflexive and transitive.\footnote{In the literature $\peq$ is usually assumed to be converse well-founded, but we shall work in a more general setting. Note that the plausibility relation need not be {\em connected}.}
\item if $B\in \sigma$ then $I(B)$ is an (arbitrary) binary relation on $W$, and\footnote{For readers familiar with the more standard literature on modal logics of belief, it is natural to wonder why we include a separate relation for our belief modalities (rather than using the plausibility ordering to define belief in the usual way).       The reason for working with this more general definition is discussed in Sections \ref{derived} and  \ref{belplaus}.  }
\item $V:\At\rightarrow \pow(W)$ is a valuation function.
\end{itemize}

We will not distinguish between $m\in\sigma$ and its denotation $I(m)$ in what follows, and we will write $E(w)$ for the set $\{X\ |\ wEX\}$.

\smallskip

A $\sigma$-structure is a {\bf $\sigma$-model} if it further satisfies the following constraints:
\begin{enumerate}
\item for each $w\in W$, $\varnothing \not\in E(w)$ and  $W\in E(w)$ and
\item if $\mathord\peq\in \sigma$ then for all $w,v,u\in W$, if $w\peq v$ and $w\in X\in E(u)$ it follows that $v \in X$ and
\item whenever $w\peq v$ and $u\mathrel B w$, it follows that $u\mathrel B v$.
\end{enumerate} 
\end{definition}

 An {\bf evidence model} is a $\sigma$-model where $\sigma=E$. We may also write simply {\bf model} instead of $\sigma$-model when the value of $\sigma$ is clear from context.
 
 The following basic assumptions are implicit in the above definition: 

\begin{itemize}

\item Sources may or may not be {\em reliable}: a subset recording a piece of evidence need not contain the actual world. Also, agents need not know which evidence is reliable.

\item The evidence gathered from different sources (or even the same source) may be jointly inconsistent.   And so, the intersection of all the gathered evidence may be empty. 

\item Despite the fact that sources may not be reliable or jointly inconsistent, they are all the agent has for forming beliefs.\footnote{Modeling sources and agents'  {\em trust} in these is quite feasible in our setting -- but we will not pursue this topic here.}
 
\end{itemize} 

 The {\em evidential state} of the agent is the set of all propositions  (i.e., subsets of $W$) identified by the agent's sources. In general, this could be any collection of subsets of $W$; but we  do impose some  mild conditions, that were stated in Constraint 1 in the above definition of our models:
 
 \begin{itemize}
 \item No evidence set is empty (evidence per se is never contradictory),
 \item The whole universe $W$ is an evidence set (agents know their `space').
 \end{itemize}
 In addition, one might expect a   `monotonicity'  assumption: 
\begin{quote} 
  If the agent has evidence $X$ and $X\subseteq Y$ then the agent has evidence $Y$.  
\end{quote}
To us, however, this is a property of propositions supported by evidence, not of the evidence itself. Therefore, we model this feature differently through the definition of our ``evidence for" modality (see Definition \ref{truthL}).

 Next, constraint 2 in the definition of our models ensures that  the agent's evidence ``coheres'' with her opinions about the relative plausibility of the states.  The intended interpretation of  $w\peq v$ is that, although the agent does not know which of $w$ or $v$ is the ``actual situation", she considers $v$ at least as plausible as $w$.  Thus, if   $X$ is evidence that the actual state may be $w$ (i.e., $w\in X$), then $X$ is also evidence for all states that the agent considers at least as plausible as $w$.   The underlying idea is that the agent uses her plausibility ordering to ``fill-out" the sets of states identified as evidence.  We will also  consider the even more constrained situation where the agent defines her plausibility ordering directly from the evidence (this is discussed in Section \ref{derived}). 
 
Finally, constraint 3 on our evidence models generalizes the usual assumption that the agent believes what holds throughout the  set of most plausible worlds:   Any state the agent believes is possible must be among the most plausible states overall, though we will not require that the converse always holds. 

\medskip
Let us define {\em truth} for formulas of $\L$ in a given model:

\begin{definition} \label{truthL} Let $\M=\langle W, E,B,\peq, V\rangle$ be an evidence model. Truth of a formula $\phi\in\L$ is defined inductively as follows: 
 \begin{itemize}
\item $\M,w\models p$ iff $w\in V(p)$\hspace{.3in} ($p\in \At$)
\item $\M,w\models \neg\phi$ iff $\M,w\not\models\phi$
\item $\M,w\models\phi\wedge\psi$ iff $\M,w\models\phi$ and $\M,w\models\psi$
\item  $\M,w\models\nc\phi$ iff  there exists $X$ such that $wEX$ and for all $v\in X$, $\M,v\models\phi$
\item  $\M,w\models \nB\phi$ if for all $v$, if  $w \mathrel B v$, then  $\M,v\models\phi$  
\item $\M,w\models \nl\phi$ iff for all $v$, if $v\seq w$, then  $\M,v\models\phi$ 
\item $\M,w\models \nA\phi$ iff for all $v\in W$, $\M,v\models\phi$ 
\end{itemize}
The truth set of $\phi$ is the set $\truth{\phi}{\M}=\{w\ |\ \M,w\models\phi\}$.   The standard logical notions of {\bf satisfiability} and {\bf validity} are defined as usual.  
 \end{definition}

 \subsection{From general models to intended models}\label{derived}

In this section, we consider an alternative way of interpreting our modal language $\L$, specializing the above ``general models'' to a special class where belief is entirely and explicitly evidence-driven.   
Rather than interpreting each modality   as a new relation on the set of states, we will now {\em derive} the relevant relations from the evidence sets.   As a result,  we can interpret the entire language $\L$ on an evidence model $\M=\langle W, E, V\rangle$.  
 
Recall that we do not assume that the collection of evidence sets $E(w)$ is closed under supersets. Also, even though evidence pieces are non-empty, their combination through the obvious operations of taking {\em intersections} need not yield consistent evidence: we allow for disjoint evidence sets, whose combination may lead (and should lead) to trouble. But importantly, even though an agent may not be able to consistently combine {\em all} of her evidence, there will be maximal collections of admissible evidence that she can safely put together to form {\em scenarios}:

 \begin{definition} \label{fip}  A {\bf $w$-scenario} is a maximal collection $\X\subseteq E(w)$ that has the fip (i.e., the finite intersection property:   for each finite subfamily $\{X_1,\ldots,X_n\}$ $\subseteq \X$ we have that
$\bigcap_{1\le i\le n} X_i\ne\varnothing$).  A collection is called a {\bf scenario} if it is a $w$-scenario for some state $w$.  
\end{definition}
Our notion   of having evidence for $\phi$ need not imply that the 
agent {\em believes}   $\phi$.  In order to believe a proposition  $\phi$, the agent must consider {\em all} her evidence for or against $\phi$.  The idea is that each $w$-scenario  represents a maximally consistent theory based on (some of) the evidence collected at  $w$. \footnote{Analogous ideas occur in semantics of conditionals \cite{kratzer,veltman} and belief revision  \cite{gardenfors,rott}.}  Note that the  definition of truth of the ``evidence for" operator   builds in monotonicity.  That is,   the agent has evidence for $\phi$ at $w$ provided there is some evidence  available at $w$ that implies $\phi$.   This motivates the following definition:

\begin{definition}
Given an evidence function $E:W\to \pow (W)$, we define $B_E\subseteq W\times W$ by $w \mathrel B_E v$ if $v\in\bigcap \mathcal X$ for some $w$-scenario $\mathcal X$.
\end{definition}

In order to derive a plausibility ordering, we  borrow a ubiquitous idea, occurring in point-set topology, but also in recent theories of relation merge   (cf. \cite{ARS,fenrong-book}): the so-called {\em specialization (pre)-order}.  Under this ordering,  $v$ is more plausible than $w$ if every set that is evidence for $w$ is also evidence for $v$. Formally, we define the following:

\begin{definition}\label{derived-plaus}
Given a evidence function $E:W\to \pow (W)$, we define $\peq_E\subseteq W\times W$ by $w \peq_E v$ if whenever $u,X$ are such that $w\in X\in E(u)$, then $v\in X$.
\end{definition}
To make this definition a bit more concrete, here is a simple illustration. 
 
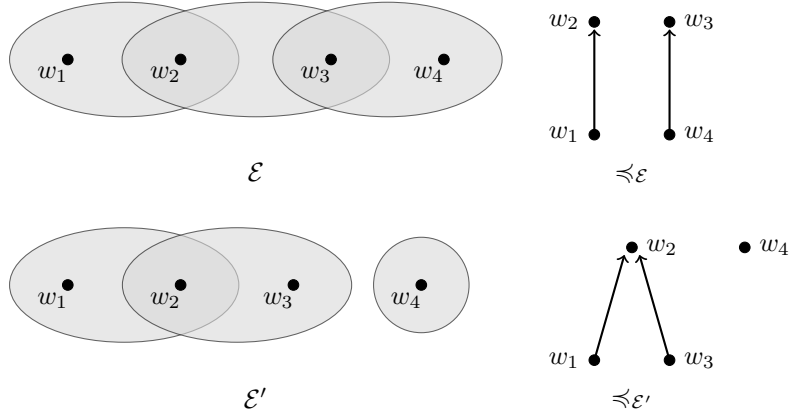
\begin{center}
\begin{tikzpicture} 

\node[ellipse,draw,fill=gray!30,opacity=0.6,minimum height=0.6in,minimum width=1.2in] at (0.75,0) {$ $}; 
\node[ellipse,draw,fill=gray!30,opacity=0.6,minimum height=0.6in,minimum width=1.4in] at (2.5,0) {$ $}; 
\node[ellipse,draw,fill=gray!30,opacity=0.6,minimum height=0.6in,minimum width=1.2in] at (4.25,0) {$ $}; 

\draw[fill] (0,0) circle (2pt);
\node at (-0.2,-0.2) {$w_1$}; 

\draw[fill] (1.5,0) circle (2pt);
\node at (1.3,-0.2) {$w_2$}; 

\draw[fill] (3.5,0) circle (2pt);
\node at (3.3,-0.2) {$w_3$}; 

\draw[fill] (5,0) circle (2pt);
\node at (4.8,-0.2) {$w_4$}; 

\node at (2.5,-1.5) {$\E$}; 

\draw[fill] (7,0.5) circle(2pt); 
\node at (6.6,0.5) {$w_2$}; 

\draw[fill] (7,-1) circle(2pt); 
\node at (6.6,-1) {$w_1$};
\path[<-,thick,draw] (7,0.4) to (7,-1); 

\draw[fill] (8,0.5) circle(2pt); 
\node at (8.4,0.5) {$w_3$}; 

\draw[fill] (8,-1) circle(2pt); 
\node at (8.4,-1) {$w_4$};
\path[<-,thick,draw] (8,0.4) to (8,-1); 
\node at (7.5,-1.5) {$\peq_\E$};

\node[ellipse,draw,fill=gray!30,opacity=0.6,minimum height=0.6in,minimum width=1.2in] at (0.75,-3) {$ $}; 
\node[ellipse,draw,fill=gray!30,opacity=0.6,minimum height=0.6in,minimum width=1.2in] at (2.25,-3) {$ $}; 
 \node[circle,draw,fill=gray!30,opacity=0.6,minimum width=0.5in] at (4.7,-3) {$ $}; 
\draw[fill] (0,-3) circle (2pt);
\node at (-0.2,-3.2) {$w_1$}; 

\draw[fill] (1.5,-3) circle (2pt);
\node at (1.3,-3.2) {$w_2$}; 

\draw[fill] (3,-3) circle (2pt);
\node at (2.8,-3.2) {$w_3$}; 

\draw[fill] (4.7,-3) circle (2pt);
\node at (4.5,-3.2) {$w_4$}; 

\node at (2.5,-4.5) {$\E'$};

 \draw[fill]  (7,-4) circle (2pt); 
 \node at (6.6,-4) {$w_1$}; 

  \draw[fill]  (8,-4) circle (2pt); 
\node at (8.4,-4) {$w_3$}; 

\draw[fill] (7.5,-2.5) circle(2pt);
\node at (7.9,-2.5) {$w_2$}; 
 
  \draw[fill]  (9,-2.5) circle (2pt); 
  \node at (9.4,-2.5) {$w_4$}; 
\path[->,thick,draw] (7,-4) to (7.4,-2.6); 
\path[->,thick,draw] (8,-4) to (7.6,-2.6); 

\node at (7.5,-4.5) {$\peq_{\E'}$}; 

\end{tikzpicture}
\end{center}

Note that the derived relation $\peq_E$ is uniform throughout the model, even though evidence itself is not. It is indeed possible to define a point-wise variant of the specialization order with $u$ as parameter, but we shall not explore this option in this paper.

Given an evidence model $\mathcal M=\langle W,E,V\rangle$, we now define the extended structure
\[\mathcal M^\vartriangle=\langle W,E,B_E,\peq_E,V\rangle.\]
We extend the truth-definitions on $\mathcal M$ to all of $\L$ by setting $\lb\cdot\rb_\mathcal M=\lb\cdot
\rb_{\mathcal M^\vartriangle}$.
If $\mo M=\langle W,E,V\rangle^\vartriangle$, we say $\mo M$ is an {\em intended model.}

We will be interested in the precise relationship between intended models $\M^\vartriangle$ and our earlier general models $\M=\langle W, E, B, \peq, V\rangle$.  This issue will be discussed in detail in the remainder of the paper, but we conclude this section with a brief remark. Notice that, in general, $\M = \langle W,E,B,\peq_E,V\rangle$ is not necessarily a model according to Definition \ref{ev-model}. The problem is that the  constraint stating that if $w \mathrel B v\peq_E u$ then $w\mathrel B u$ is not necessarily satisfied. A particularly simple example is a uniform model where $W=\{w,v\}$ and $E(w)=E(v)=\{W\}$. If we take, for example, $B=\{(w,w)\}$, then it should be clear that $w\mathrel B w\peq_E v$ yet $w\not\mathrel B v$.

However, this can never happen over an intended model:

\begin{lemma} Suppose that  $\M=\langle W, E, V\rangle $ is an evidence model, then $\M^\vartriangle$ is a model according to Definition \ref{ev-model}.  
\end{lemma}

\begin{proof} Most conditions are obvious; the only one that we need to check is that if    $u\mathrel B_E w$ and $w\peq_E v$  then $u\mathrel B_E v$.   Since $u\mathrel B_E w$, there is a scenario $\X$ such that $w\in\bigcap \X$. But since $w\peq_E v$, it follows that also $v\in\bigcap \X$, and since $\X$ was a $u$-scenario, we have that $u\mathrel B_E v$.
\end{proof}

\subsection{p-morphisms}

It is important to identify maps between structures which preserve truth-sets. In our setting, the following definition will be very useful:

\begin{definition}\label{bisim}
Given evidence models $\mo{M}_1=\langle W_1,B_1,E_1,\peq_1,V_1\rangle$ and $\mo{M}_2=\langle W_2,B_2,E_2,\peq_2,V_2\rangle$, we say a function $\pi:W_1\to W_2$ is a p-{\em morphism} if the following conditions hold:
\begin{description}
\item[$\mathsf{atoms}$] $V_1=\pi^{-1}V_2$
\item[$\mathsf{forth}_R$] if $w\mathrel B_1 v$ then $\pi(w)\mathrel B_2\pi(v)$
\item[$\mathsf{back}_R$] if $\pi(w)\mathrel B_2 u$ then there is $v\in\pi^{-1}(u)$ such that $w\mathrel B_1 v$
\item[$\mathsf{forth}_E$] if $w\mathrel E_1 X$ then there is $Y\subseteq W_2$ such that $\pi(w)\mathrel E_2 Y$ and $Y\subseteq\pi[X]$
\item[$\mathsf{back}_E$] if $\pi(w)\mathrel E_2 Y$ then there is $X$ such that $\pi[X]\subseteq Y$ and $w\mathrel E_1 X$
\item[$\mathsf{forth}_\peq$] if $w\peq_1 v$ then $\pi(w)\peq_2\pi(v)$
\item[$\mathsf{back}_\peq$] if $\pi(w)\peq_2 u$ then there is $v\in\pi^{-1}(u)$ such that $w\peq_1 v$
\end{description}

We define p-morphism between structures of smaller signature by restricting the above conditions accordingly.
\end{definition}

Then we obtain the following familiar result which we present without its straightforward inductive proof:
\begin{theorem}\label{bistheo}
If $\pi$ is a p-morphism between $\sigma$-models $\pmo{M}_1$ and $\pmo{M}_2$ and $\varphi$ is any formula over the signature $\sigma$, then
\[\truth{\varphi}{\pmo{M}_1}=\pi^{-1}[\truth{\varphi}{\pmo{M}_2}].\]
\end{theorem}

If a surjective $p$-morphism exists from $\pmo{M}_1$ to $\pmo{M}_2$, we will write $\pmo{M}_1 \gg \pmo{M}_2$. In Section \ref{representation} we will show that, given a general evidence model  $\pmo{M}$, there is an evidence model ${\pmo{M}^\triangledown}=\langle W^\triangledown,E^\triangledown,V^\triangledown\rangle$ such that $(\pmo{M}^\triangledown)^\vartriangle\gg{\pmo{M}}$.   
Thus a general evidence model may always be represented as a $p$-morphic image of an intended evidence model. This intended model $\pmo M^\triangledown$, however, will often be much larger.

\subsection{Special classes of evidence models}

The class of evidence models we have described gives the most general setting such an agent may face. However, there are natural additional assumptions one may consider:

\begin{definition} 
Suppose that $\mo{M}$ is a  $(EB\!\!\peq)$-model.   We say  $\mo{M}$ is {\bf flat} if whenever $w\mathrel E X$ there is $v\in X$ such that $w\mathrel B v$.

We say $\mo{M}$ is {\bf uniform} if $E$ and $B$ are constant (that is, given $u,v,w$, then $wBv$ if and only if $uBv$, and given $w,v$, $E(w)=E(v)$), and whenever $w\mathrel B v$ it follows that $v$ is $\peq$-maximal. In this case, we shall treat $E$ as a set $\mathcal E$ (of neighborhoods) rather than a function and $B$ as a set of points.

Finally, we say $\mo{M}$ is {\bf concise} if it is flat, uniform and, whenever $w$ is $\peq$-maximal then $w\in B$.
\end{definition}
The reason for this  terminology will show in the arguments to follow. The definition applies, {\em mutatis mutandis},  to models with different signatures (e.g., evidence models).     Lemma \ref{mfipmax} shows  that if $\mo M=\langle W,\E,V\rangle$ is a uniform intended evidence model then every point in $B_E$ is maximal, whereas if $\mo M^\vartriangle$ is flat as well, then $B_E$ coincides exactly with the set of maximal points. Note, however, that in general the property of being concise is stronger than that of merely being flat {\em and} uniform.

Flatness may seem like an odd assumption but it is actually very natural because every finite intended model is flat, though infinite models need not be. Uniformity is also natural because one often wishes to model a situation where an agent either {\em has} or does {\em not} have evidence right at the beginning, rather than obtaining different evidence at each state.


\subsection{The logics}
We now turn to logics for reasoning about distinct classes of evidence models. Our first observation is that the language $\L$ is  sensitive to flatness:

\begin{lemma}\label{flatax}
If $\mo{M}$ is a flat model, then $\mo{M}\models \nc\phi\rightarrow \pB\phi$. 
\end{lemma}
 \begin{proof}
If $X\in E(w)$ is an evidence set witnessing $\varphi$ (i.e., $X\subseteq \truth{\phi}{\M}$), then the singleton $\{X\}$ can be extended to a $w$-scenario using Zorn's Lemma, which, in flat structures, has a non-empty intersection.
\end{proof}

Meanwhile, the formula $\nc\phi \to \pB\varphi$ is not valid in general. To see this,  consider a uniform evidence model $\mo{M}_{\infty} = \langle W, E, V\rangle$  with domain $W=\mathbb N$ and evidence sets $E(w)=\mathcal E=\{[N,\infty)\ |\  N\in \mathbb{N}\}\cup\{W\}$ for each $w\in W$. The valuation is unimportant, so we may let $V=\varnothing$.  Clearly,  the only scenario on $\mo{M}_{\infty}$ is all of $\mathcal E$, but $\bigcap \mathcal E=\varnothing$. Hence $\mo{M}_{\infty}\models \nB\bot$, i.e.,  $\mo{M}_{\infty}\not\models\pB\top$; yet $\mo{M}_{\infty}\models\nc\top$ (this formula is universally valid), and we conclude  
$$\mo{M}_{\infty}\not\models\nc\top\to\pB\top.$$
 From this we get the following corollary:

\begin{corollary}\label{nonfin}
The logic of evidence models does not have the finite model property, nor does the logic of uniform evidence models.
\end{corollary}

\begin{proof} Every finite model is flat, and hence it validates $\nc\top\to\pB\top$; but as we have just shown, this formula of our language is not valid over all uniform evidence models.
\end{proof}

With this in mind, we state a list of axioms and rules for evidence logics:\\

\begin{center}
\scalebox{0.8}{
\begin{tabular}{lcr}
Name&Formula&Logic(s)\\
\hline &&\\
{\bf tautology}&all propositional tautologies&all\\\\

$\mathsf{S5}_A$&$\mathsf {S5}$ axioms for $A$&all\\\\

$\mathsf K_B$& $\mathsf {K}$ axioms for $B$&all\\\\

$\mathsf S4_\peq$& $\mathsf {S4}$ axioms for $\nl$&all\\\\

{\bf no empty evidence}&$\ps\top$&all\\\\

 {\bf pullout}&$\nc\varphi\wedge \nA\psi\leftrightarrow \nc(\varphi\wedge \nA\psi)$&all\\\\

{\bf universality}&$\nA\varphi \to [X]\phi$ for $X=E,B,\peq$&all\\\\

{\bf $E$-monotonicity}&$\displaystyle{\dfrac{\varphi\to\psi}{\nc\varphi\to\nc\psi}}$&all\\\\

\end{tabular} 
}
\scalebox{0.8}{
\begin{tabular}{lcr}
Name&Formula&Logic(s)\\
\hline &&\\

{\bf plausible evidence}& $\nc\varphi\to\nc(\varphi\wedge\nl\varphi)$&all\\\\

{\bf $B$-monotonicity}&$\nB\varphi\rightarrow\nB\nl\varphi$ &all\\\\

{\bf flatness}&$\nc\varphi\to \pB\varphi$&$\flat,\flat u$\\\\

{\bf $\bigcirc$-uniformity}&$\bigcirc\varphi\to \nA\bigcirc\varphi$&$u,\flat u$\\
&for $\bigcirc=\nB,\pB,\nc,\ps$&\\

{\bf maximality}&$\pB(\pl\alpha\wedge\beta)\rightarrow \pB(\alpha\wedge\pl\beta)$&$u,\flat u$\\\\

{\bf conciseness}&$\nB\varphi\to \pl\nl\varphi$&$\flat u$\\\\

{\bf MP}&Modus Ponens&all\\\\
{${\bf N}_A$}&$\displaystyle\dfrac{\phi}{\nA\phi}$&all
\end{tabular}
}
\end{center}

The above table gives a family of axioms which in turn define a family of logics; given a signature $\sigma$, we let $\logic^\sigma$ denote the logic which uses only those axioms and rules that fall within $\L^\sigma$ and are marked ``all''. The subscripts $\flat,u$ denote the addition of the respective axioms; note that the logic ${\sf Ev}^\sigma_{\flat u}$ includes the additional conciseness axiom which is not present in any other logic. We will denote derivability in the logic $\lambda$ by $\vdash_\lambda$, where $\lambda$ is any one of the four combinations that we may form from a given $\sigma$, that is, $\lambda\in\{\logic^\sigma,\logic^\sigma_{\flat},\logic^\sigma_{u},\logic^\sigma_{\flat u}\}$. If $\Theta$ is a (possibly infinite) set of formulas, $\Theta\vdash_\lambda\phi$ means that there is a finite $\Theta'\subseteq \Theta$ such that $\vdash_\lambda\bigwedge\Theta'\to\phi.$

These axioms are sound for their respective classes of models:

\begin{theorem}
The logics $\logic^\sigma,\logic^\sigma_{\flat},\logic^\sigma_{u},\logic^\sigma_{\flat u}$ are sound for the classes of all $\sigma$-models, all uniform $\sigma$-models, all flat $\sigma$-models and all concise $\sigma$-models, respectively.
\end{theorem}

\begin{proof}
We assume familiarity with neighborhood semantics and well-known modal logics, so we will restrict ourselves to commenting on the more unusual axioms.

Let $\mo M$ be any model; by passing to $\mo M^\vartriangle$ if necessary, we may assume without loss of generality that $\mo M=\langle W,E,B,\peq,V\rangle$.

Let us begin by checking {\bf plausible evidence}. Suppose that $w\in \lb\nc\varphi\rb_\mathcal M$, so that there is $X\in E(w)$ such that $X\subseteq\lb\varphi\rb_\mathcal M$. Pick $v\in X$ and suppose $v\peq u$. Then, $u\in X$, so that $u\in \lb\varphi\rb_\mathcal M$; since $u$ was arbitrary, $v\in\lb\varphi\wedge\nl\varphi\rb_\mathcal M$ and thus $w\in \lb\nc(\phi\wedge\nl\varphi)\rb_\mathcal M$, as claimed.

The {\bf $B$-monotonicity} axiom follows easily from the condition that if $u\peq v$ and $w\mathrel B u$, then $w\mathrel B v$.

The {\bf flatness} axiom is Lemma \ref{flatax} while {\bf conciseness} follows from Theorem \ref{flatunitheo}, which we prove later, and it gives the slightly stronger $\nB\varphi\rightarrow \nA\pl\nl\varphi$.

Finally, the {\bf maximality} axiom uses the fact that, over concise models, every element of $B$ is maximal. For indeed, if $\pB(\alpha\wedge\pl\beta)$ holds, there is some $w\in B\cap\lb \alpha\wedge\pl\beta\rb_\mathcal M$ and thus some $v\seq w$ with $v\in \lb \beta\rb_\mathcal M$. Since $w$ is maximal we must have that $v\peq w$ and that $v$ is maximal as well, so that $v\in B$ and $v$ satisfies $\pl\alpha$. But it also satisfies $\beta,$ so we have that $\pB(\pl\alpha\wedge\beta)$ holds.
\end{proof}

The weakest logic $\logic$ will be called {\em general} evidence logic, while $\logic_\flat,\logic_{\flat u}$ will be called {\em flat logics} and $\logic_u,\logic_{\flat u}$ will be called {\em uniform logics}. We will write {\em $\lambda$-consistency} for consistency over the logic $\lambda$.

\section{Belief and plausibility}\label{belplaus}

 Now we step aside, and consider a different tradition in modeling beliefs. 

A {\em plausibility} model is a tuple $\langle W, \peq, V\rangle$ where $W$ is a nonempty set, $V$ is a valuation function and   $\peq$ is a reflexive, transitive and well-founded order on $W$.   We assume the reader is familiar with these well-studied models and the modal languages used to reason about them (see \cite{vanbenthem-newbook} for details and pointers to the relevant literature).    Our evidence models  are not intended to {\em replace} plausibility models, but rather to complement them.   So, what exactly is the  relationship between these two frameworks for modeling beliefs?  The answer is subtle (see \cite[Section 5]{vBP-evidence} for some initial observations).

 On plausibility models, it is commonplace to say that an  agent believes a formula if it is  true in all the $\peq$-maximal worlds.    This makes belief definable in the language containing only the universal modality $\nA$ and the plausibility modality $[\peq]$, namely, as $\nA\langle \peq \rangle [\peq]\phi$.\footnote{This is not quite correct if there  is more than one information cell (i.e., equivalence class of the knowledge relation).  The formula $\nA\langle\peq\rangle[\peq]\phi$ means that $\phi$ is believed {\em at all states in the model}.    We are implicitly assuming that there is only one ``information cell" for the agent, namely the set of all states $W$.   In  this case, the universal modality $\nA$ is best interpreted as  ``knowledge".}   This may look like a simple technical ploy, but  Baltag and Smets \cite{baltag-smets} interpret $[\peq]\phi$ independently as ``a {\em safe belief} in $\phi$". Following  \cite{stalnaker-bi}, they show that this amounts to the beliefs the agent retains under all new true information about the actual world.\footnote{For the same notion in the computer science literature on agency, see \cite{masbook}.  Also, see \cite{stalnaker-logknowbel} for related discussions. }   Our first observation is that our belief operator $[B]$ is not similarly definable over the class of all evidence models.

\begin{lemma}
The modality $\nB$ is not definable in terms of $\nA,\nc,\nl$, even over uniform intended evidence models.
\end{lemma}

\proof
The model $\mathcal M_\infty$ from Corollary \ref{nonfin} will do. Observe that $\mathcal M^\vartriangle_\infty$ is bisimilar to the trivial model $\mathcal A^\vartriangle$ where $\mathcal A$ has one point $w$ and one evidence set $\{w\}$ (simply because all propositional variables are false). Yet $\mathcal M^\vartriangle_\infty\models \nB\bot$ while the model $\mathcal A^\vartriangle$ is finite, and hence it validates $ \nB\top$.
\endproof

The situation changes over the class of flat models. Here the following notion will be useful later:

\begin{definition}
Given a uniform evidence model $\mathcal M=\langle W,\mathcal E,V\rangle$, write $\E[w]=\{X\in\mathcal E\ |\ w\in X\}$.
\end{definition}
This construction suggests an extension to our language that we will briefly discuss before moving on to the main result of this section. While what follows can be skipped without loss of continuity, it does give a more concrete idea of the many natural kinds of evidence structure to be found in our models.

\paragraph{Digression: reliable and unreliable evidence} Suppose that  $\M=\langle W, \E,V\rangle$ is  a uniform evidence model.   Then  $\E[w]$ is the  set of evidence that is ``correct" or reliable at state $w$. Unlike the agent's full evidential state at $w$, the set of reliable evidence at $w$ can always be consistently combined, suggesting  a new modality  $[C]\phi$ meaning ``the agent's reliable evidence entails $\phi$" or, for lack of a better term,  ``$\phi$ is reliably believed".  The formal definition is: 
$$ \M,w\models[C]\phi\text{ iff  for all $v\in \bigcap \E[w]$, $\M,v\models\phi$}$$

There is one technical issue we need to address here, which has interesting consequences.  Note that the set of reliable evidence differs from state to state.  That is, even if the evidence function is assumed to be constant, the {\em reliable} evidence function   will not be constant.  This means that, unlike plain belief, even in  uniform evidence models, reliable belief does not satisfy introspection properties.  So, what type of doxastic attitude is $\Box^C$?  We do not have the space here for an extensive discussion, but here are a few comments.  
 \medskip

First, notice that $\Box^C$ validates the truth axiom ($\Box^C\phi\rightarrow\phi$), so believing something based on reliable evidence implies truth. Of course, it is only the modeler, from a third person perspective, who can actually determine what the agent believes based only on reliable evidence.   Since the agent does not have access to the actual world, she herself cannot determine which evidence is reliable and which is not.  Second, the restriction to   ``truthful" evidence suggests that reliable belief might be {\em  safe belief} on evidence models.  However, the derived plausibility ordering is typically not connected, and   on such models,  safe belief quantifies over all worlds not strictly less plausible than the current world. In particular, if $\phi$ is safely believed at a world $w$ in a non-connected plausibility model, then $\phi$ must be true at all worlds that are incomparable with $w$. Now, for a derived plausibility ordering, there are two reasons why a state $v$ may be incomparable to the actual state $w$: either there is reliable evidence not containing $v$, or there is evidence containing $v$ but not $w$.   This suggests yet another modality.

Let $\M=\langle W, E, V\rangle$ be an evidence model and define $E^U(w)=\{X\in E(w)\ |\ w\not\in X\}$.  This is the {\em unreliable} or incorrect evidence at state $w$.  The corresponding modality is $\Box^U$: ``$\phi$ follows from the unreliable evidence at $w$".   Of course, the agent cannot necessarily consistently combine this evidence, but in the formal definition we can take the {\em union} of these sets: 
  $$ \M,w\models[U]\phi\text{ iff  for all $v\in \bigcup E^U(w)$, $\M,v\models\phi$}$$
It is the  conjunction of these two operators  that corresponds to safe belief on evidence models.

These new operators are not definable in our basic language $\L$.

\begin{fact} The operators $[C]$ and $[U]$ are not definable in evidence belief language $\L_{ABE\peq}$. 
\end{fact}

\begin{proof}  Let $\M$ and $\M'$ be uniform evidence models with the evidence sets: 

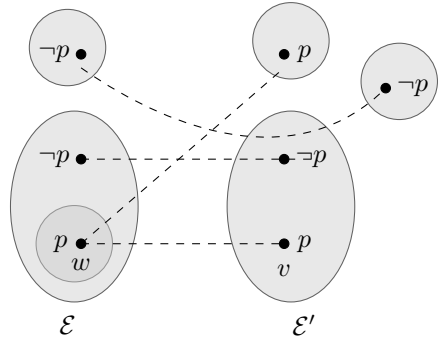
\begin{center}

\begin{tikzpicture}[scale=0.9]
\node[circle,draw,fill=gray!30,opacity=0.8,minimum width=0.4in] at (-0.1,0) {$ $}; 
\node[ellipse,draw,fill=gray!30,opacity=0.6,minimum width=0.6in, inner sep=6mm, minimum height=1in] at (-0.1,0.55) {$ $}; 

\node[circle, draw,fill=gray!30,opacity=0.6,minimum width=0.4in] at (-0.2,2.9) {$ $}; 
\draw[fill] (0,0) circle (2pt); 
\node at (-0.3,0) {$p$}; 
\node at (0,-0.3) {$w$}; 

\draw[fill] (0,1.25) circle (2pt); 
\node at (-0.4,1.25) {$\neg p$}; 

\draw[fill] (0,2.8) circle(2pt); 
\node at(-0.4,2.8) {$\neg p$}; 

\node[ellipse,draw,fill=gray!30,opacity=0.6,minimum width=0.6in, inner sep=6mm, minimum height=1in] at (3.1,0.55) {$ $}; 
 
\node[circle,draw,fill=gray!30,opacity=0.6,minimum width=0.4in] at (3.1,3) {$ $}; 
\node[circle,draw,fill=gray!30,opacity=0.6,minimum width=0.4in] at (4.7,2.4) {$ $}; 

\draw[fill] (3,2.8) circle (2pt); 
\node at (3.3,2.8) {$p$};

\draw[fill] (4.5,2.3) circle (2pt); 
\node at (4.9,2.3) {$\neg p$}; 

\draw[fill] (3,0) circle (2pt); 
\node at (3.3,0) {$p$}; 
\node at (3,-0.4) {$v$};

\draw[fill] (3,1.25) circle (2pt); 
\node at (3.4,1.25) {$\neg p$}; 

 \path[-,dashed,draw] (0,0) to (3,0); 
 \path[-,dashed,draw] (0,0) to (3,2.6); 

\path[-,dashed,draw] (0,1.25) to (3.4,1.25); 
 \path[-,dashed,draw] (0,2.6) to[in=225,out=325] (4.5,2.3); 

\node at (-0.2,-1.2) {$\E$}; 
\node at (3.3, -1.2) {$\E'$}; 
\end{tikzpicture}
\end{center}
The dashed line is a   $p$-morphisim, so  $\M \gg \M'$.        So, $\M,w$ and $\M',v$ satisfy the same formulas of $\L_{ABE\peq}$.  However, we have\footnote{In general,  it need not be the case that the agent reliably believes $\phi$ iff the agent unreliably believes $\neg\phi$.   We do not discuss the logic of these operators here, but the reader might note that the set of reliable evidence  at a state is a {\em filter} while the set of unreliable evidence is an {\em ideal}.} $\M,w\models[C] p \wedge [U] \neg p$ while $\M',v\models\neg[C]p\wedge\neg[U] \neg p$.  
\end{proof}

\medskip

This ends our digression on reliable and unreliable evidence as an example of richer languages supported by our evidence models. This theme of language extensions will return in Section 7, but for now, we return to the earlier source of variety, the existence of natural subclasses of evidence models.
\medskip
 
We are now ready to clarify the relationship between uniform,  flat evidence models and plausibility models.  It will  be helpful to start with the following observation: 

\begin{lemma}\label{mfipmax}
If $\mathcal M=\langle W,\mathcal E,V\rangle$ is a uniform evidence model, then every $w\in B_\mathcal E$ is maximal and every non-empty scenario $\mathcal X$ is of the form $\mathcal E[w]$ for some $w \in B_\mathcal E$, so that
\[\bigcap \mathcal X=\{v:w\peq_\mathcal E v\text{ and }v\peq_\mathcal E w\}.\]

If $\mathcal M$ is flat as well as uniform, then $B_\mathcal E$ is exactly the set of maximal points and every collection of the form $\mathcal E[w]$ with $w$ $\peq_\mathcal E$-maximal is a scenario.
\end{lemma}

\proof
First assume that $\mathcal M$ is a uniform evidence model and $w\in\bigcap \mathcal X$ for some scenario $\mathcal X$. It follows that $\mathcal X\subseteq \mathcal E[w]$ (since $w$ lies in every element of $\mathcal X$), and by maximality of $\mathcal X$, $\mathcal X=\mathcal E[w]$. But if we had $w\prec_\mathcal E v$ for some $v$, then clearly $\mathcal E[w]\subsetneq\mathcal E[v]$, which would contradict the maximality of $\mathcal X$ and thus $w$ is $\peq_\mathcal E$-maximal.

Furthermore, the elements of $\bigcap \mathcal E[w]$ are precisely those $v$ such that $w\peq_\mathcal E v$ (by definition of $\peq_\mathcal E$ and $\mathcal E[w]$), but from the maximality of $w$ we have that $v\peq_\mathcal E w$ as well.

For the other direction, assume $\mathcal M$ is flat and uniform and $w$ is maximal. By Zorn's lemma, $\mathcal E[w]$ can be extended to a scenario $\mathcal X$. By flatness, we have that $\bigcap \mathcal X\not=\varnothing$. If $w\not\in \bigcap \mathcal X$, there is some $v\in\bigcap \mathcal X\subseteq \bigcap\mathcal E[w]$. But then $v$ lies in every evidence set containing $w$, yet there is some $Y\in \mathcal X$ with $v\in Y$ and $w\not\in Y$, so by definition $w\prec_\mathcal E v$, which contradicts the maximality of $w$. We conclude that $\mathcal X=\bigcap \mathcal E[w]$, and thus $w\in B_\mathcal E$.
\endproof

With this, we may describe the plausibility orders of evidence models.

\begin{definition}
Let $\peq$ be a plausibility order over $W$. Say $D\subseteq W$ is {\em directed} if any two elements of $D$
have an upper bound in $D$.

A plausibility order $\peq$ satisfies the {\em boundendess condition} if every directed set $D$ has an upper bound (not necessarily in $D$).
\end{definition}

\begin{lemma}\label{flatisdir}
If an evidence model is flat, then its derived plausibility relation satisfies the boundedness condition.
\end{lemma}

\proof
Assume that $\mathcal M$ is flat and let $D$ be any directed set. Consider the family $\mathcal Y=\{Y:Y\cap D \not=\varnothing\}$. We claim that $\mathcal Y$ has the fip. Indeed, let $Y_1,\hdots, Y_n\in \mathcal Y$ and $y_i\in Y_i \cap D$. Let $y$ be an upper bound for all $y_i$ (it is an easy exercise to show that directed sets have upper bounds for finite subsets). Then, by definition, $y\in \bigcap_{i\leq n}Y_i$.

Thus $\mathcal Y$ can be extended to a scenario $\mathcal X$. Since $\mathcal W$ is flat, then $\bigcap\mathcal X$ is non-empty. But every $w\in\bigcap\mathcal X$ is an upper bound for $D$.
\endproof

\begin{corollary}\label{CorFlatMax}
If $\mathcal M$ is flat and $\peq$ is its derived plausibility relation then for every $w$ there is $v\seq w$ such that $v$ is maximal.
\end{corollary}

\proof
If $\peq$ is directed it satisfies the conditions of Zorn's lemma, and we may use it to find maximal elements above any given $w$.
\endproof

Now we are ready to show that belief is definable over flat models:

\begin{theorem}\label{flatunitheo}
Over the class of uniform evidence models with derived plausibility relation, $\nA\pl\nl\varphi$ implies $\nB\varphi$.

Over the class of models that are moreover flat, they are equivalent.
\end{theorem}

\proof
First assume that $\nA\pl\nl\varphi$ and let $\mathcal X$ be any scenario. Let $w\in\bigcap\mathcal X$. By Lemma \ref{mfipmax}, $w$ is maximal; but since it satisfies $\pl\nl\varphi$, it follows that $w$ satisfies $\varphi$.

Now assume that the model is flat and $\nB\varphi$ holds and let $w\in W$. We have to show that there is $v\seq w$ satisfying $\nl\varphi$. Use Corollary \ref{CorFlatMax} to find a maximal $v\seq w$. By Lemma \ref{mfipmax} we have that $v$ lies in $\bigcap \mathcal X$ for some scenario $\mathcal X$. Meanwhile, if $u\seq v$ then we also have $u\in\bigcap\mathcal X$, and by the assumption that $\nB\varphi$ holds, $u$ also satisfies $\varphi$, i.e. $v$ satisfies $\nl\varphi$, as desired.
\endproof

\section{The representation theorem}\label{representation}

It is generally convenient to work with models that are not necessarily intended, since it is easier to control accessibility relations than scenarios. Fortunately, as we shall show in this section, the two classes of models are equivalent with respect to our logics. For indeed, while it is not the case that every model is intended, we do have that every model is the $p$-morphic image of an intended model.

Our goal is then to define an intended model $\mo M^\triangledown$ given an arbitrary model $\mo M$. First, let us define our domain, which depends on the underlying logic. Below, recall that \[B[W]=\{w\in W:\exists v\in W\, \text{ such that } v\mathrel B w\}.\]

\begin{definition}
Let $\mathcal M=\langle W,E,B,\peq,V\rangle$ be a model. Define the following sets:

\begin{itemize}
\item $W^\triangledown=W^\triangledown_u$ is the set of all tuples $(w,x,f)$ where $w\in W$, $x\in B[W]\cup \mathbb N$ and $f:W\to\{0,1\}$.

\item $W^\triangledown_\flat$ is the set of all $(w,x,f)\in W^\triangledown$ with $x\in B[W]$.

\item $W^\triangledown_{\flat u}$ is the set of all $(w,0,f)$ such that $w\in W$ and $f:W\to \{0,1\}$.
\end{itemize}

Define $\pi(w,x,f)=w$.
\end{definition}

These will be the domains of our intended models. Note that $W^\triangledown_\lambda$ is defined even if $\mathcal M$ is not a $\lambda$-model, but in this case it will generally not be very useful. We remark that the parameter $0$ in the concise case is simply a dummy parameter which we use for uniformity with the other cases.

Let us begin with a simple observation about the size of $W^\triangledown$:

\begin{lemma}
The set $W^\triangledown_\lambda$ is finite if and only if $\mathcal M$ is finite and $\lambda$ is flat.
\end{lemma}

Before we define evidence sets, which are somewhat more elaborate, let us define the accessibility relations.


\begin{definition}
Given a model $\mathcal M=\langle W,E,B,\peq,V\rangle$ and an evidence logic $\lambda$, define a relation $\peq^\triangledown_\lambda$ on $W^\triangledown_\lambda$ by $(w,x,f)\peq^\triangledown_\lambda (v,y,g)$ if and only if $w\peq v$, either $x=y\in W$ or $x\leq y\in \mathbb N$, and $f\upharpoonright\mathop\uparrow v=g\upharpoonright\mathop\uparrow v$, where $\uparrow v=\{v'\ |\ v\peq v'\}$ is the {\em upset} of $v$.
\end{definition}

The function $\pi$ as we have defined it is a $p$-morphism on the restricted structures:

\begin{lemma}
The function $\pi:W^\triangledown_\lambda\to W$ is a surjective p-morphism with respect to the signature $\peq$.
\end{lemma}

\begin{proof}
The ${\sf forth}$ and ${\sf atoms}$ conditions are straightforward from the definitions so we focus on ${\sf back}_\peq$. But if $w\peq v$ and $w=\pi(w,x,f)$, then $(w,x,f)\peq^\triangledown_\lambda(v,x,f)$ and $\pi(v,x,f)=v$.
\end{proof}

We can now define the neighborhoods we shall use.

\begin{definition}
We define the following sets:
\begin{enumerate}


\item Given $v\in W$, $f:W\to \{0,1\}$ and $X\subseteq W$, define $X^f(v)$ to be the set of all $(w,x,g)\in W^\triangledown_\flat$ such that $x=v$ and $g\upharpoonright\mathop\uparrow w=f\upharpoonright\mathop\uparrow w$.

\item Given $n\in \mathbb N$, $f:W\to \{0,1\}$ and $X\subseteq W$, define $X^f_\lambda(n)$ to be the set of all $(w,m,g)\in W^\triangledown_\lambda$ such that $m\geq n$ and $g\upharpoonright\mathop\uparrow w=f\upharpoonright\mathop\uparrow w$.
\end{enumerate}

Our evidence relation will be divided into several parts:

\begin{enumerate}
\item Given $(w,x,f)\in W^\triangledown$ and $v\in W$, let $E^v(w,x,f)$ be the collection of all sets of the form $X^g(v)$ such that $w\mathrel E X$ and $g\upharpoonright\uparrow v=f\upharpoonright\uparrow v$.

\item Given $(w,x,f)\in W^\triangledown$, $n\in \mathbb N$, let $E^n_\lambda(w,x,f)$ be the collection of all sets of the form $X^g_\lambda(n)$ such that $w\mathrel E X$ and $X$ is not flat for $w$. Note that $E^n_\flat(w,x,f)=\varnothing$.
\end{enumerate}

If $\mathcal M=\langle W,\mathcal E, B,\peq\rangle$ is uniform, define

\begin{enumerate}
\item Given $v\in W$, $\mathcal E^v$ to be the collection of all sets of the form $X^g(v)$ such that $v\in X\in\mathcal E$ and $g(v')=1$ for all $v'\seq v$.

\item Given $n\in \mathbb N$, $\mathcal E^n_u$ to be the collection of all sets of the form $X^g_\lambda(n)$ such that $X\in\mathcal E$ and $X$ is not flat.

\item Define $\mathcal E^0_{\flat u}$ to be the collection of all sets of the form $X^g_\lambda(0)$ such that $X\in\mathcal E$.
\end{enumerate}

For non-uniform $\lambda$ let
\[E^\triangledown_\lambda(w,x,f)= \{W^\triangledown_\lambda\}\cup \bigcup_{w\mathrel B v}E^v(w,x,f)\cup \bigcup_{n\in\mathbb N}E^n_\lambda(w,x,f).\]

Finally, let
\[\mathcal E^\triangledown_u=\{W^\triangledown_u\}\cup\bigcup_{v\in B}\mathcal E^v\cup\bigcup_{n\in\mathbb N}\mathcal E^n_u\text{ and }\mathcal E^\triangledown_{\flat u}=\{W^\triangledown_{\flat u}\}\cup\mathcal E^0_{\flat u}.\]
\end{definition}

We remark that $E^\triangledown_\lambda(w,x,f)$ does not depend on $x$ but we write it this way to conform to the definition of an evidence relation. It will take some work to check that the belief relation is intended, but less so to check that plausibility is.

\begin{lemma}\label{peqlemm}
If $\mo M=\langle W,E,B,\peq,V\rangle$ is any model then $\peq^\triangledown_\lambda=\peq_{E^\triangledown}$.

If moreover $\mo M$ is uniform then $\peq^\triangledown_\lambda=\peq_{\mathcal E^\triangledown}$ for uniform $\lambda$.
\end{lemma}

\begin{proof}
First let us check that $\mathord\peq^\triangledown_\lambda\subseteq \mathord\peq_{E^\triangledown}$. Suppose that $(w,x,f)\peq^\triangledown_\lambda (v,y,g)$ and $(w,x,f)\in X\in E^\triangledown(u,z,h)$. If $X=W^\triangledown$ there is nothing do do, but otherwise consider three cases:

\begin{enumerate}
\item If $X=Y^{h'}(u')$ with $u\mathrel B u'$, then we must have $x=y=u'$, $f$ agrees with $h'$ on $\mathop\uparrow w$ and hence $g$ does on $\mathop\uparrow v$. It follows that $(v,y,g)\in Y^{h'}({u'})$.

\item If $X=Y^{h}({k})$, as before the condition on $g$ is satisfied and since $x\leq y$ we also have  $(v,y,g)\in Y^h(k)$.
\end{enumerate}

For the other direction, if $(w,x,f)\not\peq^\triangledown_\lambda (v,y,g)$, consider now three cases.
\begin{enumerate}
\item If $w\not\peq v$, then define $h$ to be identical to $f$ except that $h(v)\not= g(v)$ and define $X$ as follows: if there is $u$ such that $w\mathrel Bu$, then $X=W^h(u)$; otherwise, $W$ is not flat for $w$ and we choose $X=W^h_\lambda(0)$. In either case $X\in E(w,x,f)$, $f$ agrees with $h$ on $\mathop\uparrow w$ so $(w,x,f)\in X\in E(w,x,h)$ yet $(v,y,g)\not\in X$.

\item If $x\not=y$ and at least one of them belongs to $W$ (for example, $x$), find $u$ such that $u\mathrel B x$. Then, $(w,x,f)\in W^f(x)\in E(u,x,f)$ but $(u,y,g)\not \in W^f(x)$. The cases where either $y\in W$ but $x\in \mathbb N$ or $x>y\in\mathbb N$ are very similar.

\item If $f\upharpoonright\mathop\uparrow v\not=g\upharpoonright\mathop\uparrow v$ and $x\in W$, let $x'$ be such that $x'\mathrel B x$ if $x\in W$. Setting $X= W^f(x')$ we see that $(w,x,f)\in X\in E(w,x,f)$ but $(v,y,g)\not\in W^f(x')$. If instead $x\in\mathbb N$, we can set $X=W^f(0)$.

\end{enumerate}
We conclude that $\peq^\triangledown_\lambda=\peq_{E^\triangledown}$.
\smallskip
The uniform cases are very similar (and simpler).
\end{proof}

We can now define our intended models:
\begin{definition}
Let $\mathcal M=\langle W,E,B,\peq,V\rangle$ be a $\lambda$-model.

Define $V^\triangledown_\lambda=W^\triangledown _\lambda\cap\pi^{-1}V$ and set
\[\mathcal M^\triangledown_\lambda=\langle W^\triangledown_\lambda, E^\triangledown_\lambda, V^\triangledown_\lambda\rangle.\]
\end{definition}

These models are convenient because scenarios are easy to identify. Below, let $E^z_\lambda(w,x,f)=E^z(w,x,f)$ if $\lambda$ is not uniform, $E^z_u(w,x,f)=\mathcal E^z$.

\begin{lemma}\label{scenariolemm}
Given a non-concise $\lambda$, a $\lambda$-model $\mathcal M=\<W,E,B,\peq,V\>$ and $\mathcal M^\triangledown_\lambda=\langle W^\triangledown_\lambda,E^\triangledown_\lambda, V^\triangledown_\lambda\rangle$:
\begin{enumerate}
\item for every set $\mathcal X$ of the form
\begin{equation}\label{form1}
\mathcal X=\{W^\triangledown_\lambda\}\cup E^v_\lambda(w,x,f),
\end{equation}
where $w\mathrel B v$, $\mathcal X$ is a $(w,x,f)$-scenario, and $\pi\left(\bigcap \mathcal X\right)=\mathop\uparrow v$;
\item if $\mathcal X$ is a $(w,x,f)$-scenario that is not of the above form, then
\begin{equation}\label{form2}
\mathcal X\subseteq \{W^\triangledown_\lambda\}\cup \bigcup_{n\in\mathbb N}E^n_\lambda(w,x,f)
\end{equation}
and $\bigcap \mathcal X=\varnothing$.
\end{enumerate}
\end{lemma}

\begin{proof} We shall only consider the non-uniform case; the uniform one is analogous.

If $(v,y,g)\in X\in E^z(w,x,f)$, then either $y=z\in W$ or $y,z\in\mathbb N$. It follows that if $X,Y\in \mathcal X$ are such that $X\in E^z(w,x,f)$ and $Y\in E^{z'}(w,x,f)$, either $z=z'\in W$, $z,z'\in \mathbb N$, or $X\cap Y=\varnothing$, where the latter is impossible if $\mathcal X$ is a scenario.

Hence every scenario is either contained in a collection of the form \eqref{form1} or is of the form \eqref{form2}. Let us first check that collections of the form \eqref{form1} are indeed scenarios and have the appropriate intersection.

So assume that $\mathcal X$ is of form \eqref{form1} for some $v,f$ with $w\mathrel B v$. First note that if $v\peq u$ then $(u,v,f)\in\bigcap\mathcal X$. For clearly, $(u,v,f)\in W^\triangledown_\lambda$ and every set in $E^v(w,x,f)$ is of the form $X^v_g$ for some $g$ which agrees with $f$ on $\mathop\uparrow v$, thus on $\mathop\uparrow u$, which means that $(u,v,f)\in X^f(v)$. It follows that $\mathop\uparrow v\subseteq\pi\left(\bigcap\mathcal X\right)$.

Moreover, if $v\not\peq u$, then any element of $\pi^{-1}(u)$ is of the form $(u,u',g)$. Define a function $h$ which is identical to $f$ except that $f(u)\not= g(u)$. Then, by definition $(u,u',g)\not\in W^h(v)$, yet $w\mathrel E W$ and $h$ agrees with $g$ on $\mathop\uparrow v$ so that $W^h(v)\in E^v(w,x,f)$ and $(u,u',g)\not\in \bigcap \mathcal X$. Since $u',g$ were arbitrary, we conclude that $u\not\in \pi\left(\bigcap\mathcal X\right)$. Thus, $\mathop\uparrow v\supseteq\pi\left(\bigcap\mathcal X\right)$, and our claim follows.

Now let us check that if $\mathcal X$ is of form \eqref{form2}, then it has an empty intersection. We begin with three simple observations:
\begin{enumerate}

\item $\mathcal X$ contains at least one element of the form $X^f(n)$ with $n\in\mathbb N$. For indeed, as discussed above, if $\mathcal X$ is not of form \eqref{form1} then it does not intersect any $E^v(w,x,f)$ with $v\in W$, and the collection $\{W^\triangledown_\lambda\}$ by itself cannot be a $(w,x,f)$-scenario, for either $w\mathrel B v$ for some $v$, in which case $\{W^\triangledown_\lambda,W^f(v)\}$ would still have a non-empty intersection (it includes $(v,v,f)$), or there is no such $v$ so that $W$ is not flat for $x$, and $\{W^\triangledown_\lambda,W^f_\lambda(0)\}$ would be a proper extension with a non-finite intersection (it includes $(w,0,f)$).

\item  Any element of $\bigcap \mathcal X$ must be of the form $(v,n,g)$ with $n\in\mathbb N$ (for $X^f(n)$ contains only elements of this form).

\item If $n<k$ and $(v,n,g)\in X\in\mathcal X$ it easily follows that $(v,k,g)\in X$ by case-by-case inspection.
\end{enumerate}

With this, towards a contradiction we assume that $(v,n,g)\in \bigcap\mathcal X$ and pick $X^f(m)\in\mathcal X$. Thus, $(v,n+1,g)\in X^f({n+1})\cap\bigcap\mathcal X$, from which it follows that $(v,n,g)\not\in X^f({n+1})\cap\bigcap\mathcal X\not=\varnothing$ and thus $X^f({n+1})\not\in\mathcal X$; but this contradicts the maximality of $\mathcal X$, and we conclude that there may be no such point in the intersection and hence it must be empty.
\end{proof}

An immediate corollary of the above is the following:

\begin{corollary}
If $\mathcal M$ is flat, then so is $\mathcal M^\triangledown_\flat$
\end{corollary}

\begin{proof}
By definition, $\mathcal M^\triangledown_\flat$ cannot contain any scenarios of form \eqref{form2}, hence every scenario must be of form \eqref{form1} and thus have a non-empty intersection.
\end{proof}

We may now collect our results into the following:

\begin{theorem}\label{reptheo}
Given a $\lambda$-model $\mathcal M$, there exists an evidence model $\mathcal M^\triangledown_\lambda=\langle W^\triangledown_\lambda, E^\triangledown_\lambda, V^\triangledown_\lambda\rangle$ such that $\mathcal M$ is a p-morphic image of $(\mathcal M_\lambda^\triangledown)^\vartriangle$.

Further, if $\mathcal M$ is (i) flat and finite, (ii) flat, (iii) uniform or (iv) finite and concise, respectively, then so is $\mathcal M^\triangledown$.
\end{theorem}

\begin{proof}
We have explicity stated and proven most required properties. One exception is the $B$-clauses of a p-morphism, but this follows immediately from Lemma \ref{scenariolemm}, and in the concise case from Lemma \ref{mfipmax}, using the fact that finite models are flat, so that $B_{\mathcal E^\triangledown_{\flat u}}$ is precisely the set of maximal worlds.
\end{proof}

\section{Syntactic properties of evidence logics}\label{syntactic-properties}

We now turn our attention to establishing our main completeness results. In order to do this, we shall collect in this section the preliminary syntactic results we shall need. Recall that if $\Theta$ is a (possibly infinite) set of formulas, $\Theta\vdash_\lambda\phi$ means that there is finite $\Theta'\subseteq \Theta$ such that $\vdash_\lambda\bigwedge\Theta'\to\phi.$ In this section and the next we will work exclusively over the full signature $\sigma=EBA\mathord\peq$.

\begin{lemma}\label{verybasic} Let $\Gamma$ be a set of formulas and $\lambda$ an evidence logic over $\sigma=EBA\mathord\peq$.
\begin{enumerate}
\item If $\varphi\wedge\nl\varphi, \nA\Gamma\vdash _{\lambda}\psi$, then
\[\nc\varphi,\nA\Gamma\vdash_{\lambda} \nc\psi.\]
\item If $\Phi,\nA\Gamma\vdash_{\lambda} \psi$, then \[\nB\Phi,\nA\Gamma\vdash_{\lambda} \nB\psi.\]
\item If $\Phi,[\peq]\Delta,\nA\Gamma\vdash_{\lambda} \psi$, then \[[\peq]\Phi,[\peq]\Delta,\nA\Gamma\vdash_{\lambda} [\peq]\psi.\]
\end{enumerate}
\end{lemma}

\begin{proof} Without loss of generality we can assume $\Gamma,\Phi$ to be finite, since in general we have that $\Theta\vdash_\lambda\alpha$ if and only if for some finite $\Theta'\subseteq \Theta$, $\Theta'\vdash_\lambda\alpha$. Note further that 
\[\bigwedge \nA\Gamma\leftrightarrow \nA \bigwedge \Gamma\]
is derivable in $\mathsf{S5}$, so we can replace $\nA\Gamma$ by a single formula $\nA\gamma$.

\paragraph{1}
If $\varphi\wedge\nl\varphi,\nA\gamma\vdash_\lambda \psi$, then
\[\vdash_\lambda (\varphi\wedge\nl\varphi)\wedge \nA\gamma\to\psi.\]

By the monotonicity rule
\[\vdash_\lambda\nc\left((\varphi\wedge\nl\varphi)\wedge \nA\gamma\right)\to\nc\psi.\]

Applying the {\bf pullout} axiom we get
\[\vdash_\lambda\nc(\varphi\wedge\nl\varphi)\wedge \nA\gamma\to\nc\psi,\]
that is,
\[\nc(\varphi\wedge\nl\varphi),\nA\gamma\vdash_\lambda\nc\psi.\]
Finally, using {\bf plausible evidence} we see that
\[\nc\varphi,\nA\gamma\vdash_\lambda\nc\psi.\]

\paragraph{2} If $\Phi,\nA\gamma\vdash_\lambda \psi$, since $\nB$ is a normal operator, we also have
\[\nB\Phi,\nB \nA\gamma\vdash_\lambda \nB\psi.\]

But since $\nA\alpha\to\nA\nA\alpha$ and $\nA\nA\alpha\to \nB\nA\alpha$ are axioms, we also get
\[\nB\Phi,\nA\gamma\vdash_\lambda \nB\psi.\]

\paragraph{3} As before we may assume that $\Delta$ is finite. Also, $[\peq]$ is a normal operator and we obtain 
\[[\peq]\Phi,[\peq][\peq]\Delta,[\peq] \nA\gamma\vdash_\lambda [X]\psi\]
and, from this,
\[[\peq]\Phi,[\peq][\peq]\Delta, \nA\gamma\vdash_\lambda [X]\psi.\]

Finally, we use the transitivity of $[\peq]$ to obtain 
\[[\peq]\Phi,[\peq]\Delta, \nA\gamma\vdash_\lambda [X]\psi,\]
as desired.
\end{proof}

It will be very useful to draw some dual conclusions from the above result. Given a set of formulas $\Phi$, define $\Phi^\nA$ as $\cbra \psi \ | \ \nA\psi\in \Phi\cket$, and similarly for the other modalities. For a set of formulas $\Gamma$ and a signature $\sigma$ containing $A$, if $\lambda\in\{\logic^\sigma_{},\logic^\sigma_{\flat}\}$, define $\Gamma^\lambda=\nA\Gamma^\nA\cup \pA\Gamma^{\pA};$ if $\lambda\in\{\logic^\sigma_{u},\logic^\sigma_{\flat u}\}$, define
\[\Gamma^\lambda=\bigcup \big\{\bigcirc \Gamma^\bigcirc:\bigcirc=[X],\langle X\rangle\text{ with }X\in\{E,B,A\}\big\}.\]

\begin{lemma}\label{cons} Let $\Gamma$ be a set of formulas and $\lambda$ an evidence logic over $\sigma=EBA\mathord\peq$.

\begin{enumerate}
\item If $\Gamma$ is $\lambda$-consistent, $\alpha\in\Gamma^\nc$ and $\delta\in \Gamma^\ps$, then $\cbra\alpha\wedge \nl\alpha,\delta\cket\cup \Gamma^\lambda$ is $\lambda$-consistent.\label{cons1}

\item If $\Gamma$ is $\lambda$-consistent and $\psi\in\Gamma^{\pB}$, then $\cbra\psi\cket\cup\Gamma^{\nB}\cup\Gamma^\lambda$ is $\lambda$-consistent.

\item If $\Gamma$ is $\lambda$-consistent and $\psi\in\Gamma^{\langle \peq\rangle}$, then $\cbra\psi\cket\cup[\peq]\Gamma^{[\peq]}\cup\Gamma^\lambda$ is $\lambda$-consistent.

\item If $\lambda\in\{\logic^\sigma_{\flat },\logic^\sigma_{\flat u}\}$, $\Gamma$ is $\lambda$-consistent and $\psi\in\Gamma^\nc$, then \[\cbra\psi\wedge\nl\psi\cket\cup\Gamma^\nB\cup \Gamma^\lambda\]
is $\lambda$-consistent.
\end{enumerate}
\end{lemma}

\begin{proof} Assume $\Gamma$ is $\lambda$-consistent. As before, we may assume without loss of generality that $\Gamma$ is finite.

Note that we have $\vdash_\lambda \bigwedge\Gamma^\lambda\leftrightarrow \nA\bigwedge\Gamma^\lambda$ independently of $\lambda$; over the non-uniform case, this follows from the axiom $\pA\varphi\to \nA\pA\varphi$, while over the uniform case (i.e., if $\lambda\in\{\logic_{u},\logic_{\flat u}\}$), this is because $\bigcirc\gamma$ is equivalent to $\nA\bigcirc\gamma$ over $\lambda$ for any modality $\bigcirc=[X],\langle X\rangle$ with $X\in\{E,B,A\}$. Thus we may replace $\Gamma^\lambda$ by a single formula $\nA\gamma$ equivalent to $\bigwedge\Gamma^\lambda$ over $\lambda$.

\paragraph{1} Suppose otherwise. Then, we would have 
\[\alpha\wedge\nl\alpha,\nA\gamma\vdash_{\lambda}\neg\delta.\]
Thus by Lemma \ref{verybasic}.1,
\[\nc\alpha,\nA\gamma\vdash_{\lambda}\nc\neg\delta,\]
so that
\[\{\neg\nc\neg\delta,\nc\alpha, \nA\gamma\}\]
is $\lambda$-inconsistent.   But this is a subset of $\Gamma$, contradicting our assumption that $\Gamma$ is $\lambda$-consistent.

\paragraph{2} Suppose $\psi\in \Gamma^{\pB}$.  If the claim were false, then we would have

\[\Gamma^{\nB},\nA\gamma\vdash_{\lambda}\neg\psi,\]
and hence, by Lemma \ref{verybasic}.2,
\[\nB\Gamma^{\nB},\nA\gamma\vdash_{\lambda} \nB\neg\psi.\]
Thus
\[\cbra\neg \nB\neg\psi, \nA\gamma\cket\cup \nB\Gamma^{\nB}\]
would be $\lambda$-inconsistent, but this contradicts the consistency of $\Gamma$.

\paragraph{3} This is analogous to the previous item but uses Lemma \ref{verybasic}.3.

\paragraph{4}
If $\psi\in \Gamma^{\nc}$, the claim follows from item 2 above using the derivability of $\nc\psi\to \pB(\psi\wedge\nl\psi)$ by {\bf plausible evidence} and {\bf flatness}.
\end{proof}

\section{Completeness}\label{completeness}

Throughout this section we will work with logics in the full signature. Recall that by an {\em evidence logic} we will mean exclusively a logic of one of the forms $\logic^\sigma,\logic^\sigma_\flat,\logic^\sigma_u,\logic^\sigma_{\flat u}$. A {\em $\lambda$-type} is a maximal $\lambda$-consistent set of formulas. We state the well-known Lindenbaum lemma without proof:

\begin{lemma}
If $\Gamma$ is $\lambda$-consistent then there is a $\lambda$-type $\Delta$ such that $\Gamma\subseteq\Delta$.
\end{lemma}

Of course, in a given model all points must satisfy the same universal formulas, so it is convenient to consider the collection of all such types. 
\begin{definition}  Given $\lambda$ with signature $\sigma=EBA\mathord\peq$ and a $\lambda$-type $\Phi$, we define
\begin{enumerate}
\item $\mathrm{type}^A_{\lambda}(\Phi)$ to be the set of all $\lambda$-types $\Psi$ such that $\Psi^\lambda=\Phi^\lambda$;
\item $\mathrm{type}^B_{\lambda}(\Phi)=\cbra \Psi\in\mathrm{type}^A_{\lambda}(\Phi)\ |\ \Phi^\nB\subseteq \Psi\cket.$
\end{enumerate}
\end{definition}

Now we need to define evidence sets on our extended evidence models.

\begin{definition}
Given a $\lambda$-type $\Phi$ with  $\nc\alpha\in \Phi$, we define the {\em $\alpha$-neighborhood of $\Phi$,} denoted $\mathcal N^{\Phi}_{\lambda}(\alpha)$, as
\[\cbra \Psi\in\mathrm{type}^A_{\lambda}(\Phi) \ |\ \alpha,\nl\alpha\in\Psi\cket.\]
\end{definition}
Observe that $\mathcal N^{\Phi}_{\lambda}(\alpha)$ depends only on $\Phi^\lambda$ and not on the whole of $\Phi$. We are now ready to define a canonical extended evidence model :

\begin{definition} Let $\lambda$ be an evidence logic over $\sigma=EAB\mathord\peq$ and $\Phi$ be a $\lambda$-type.  The {\bf $\lambda$-canonical extended evidence model  for $\Phi$} is the extended evidence model  $\pmo{M}_\lambda(\Phi)= \<W_\lambda^\Phi,E_\lambda^\Phi,B_\lambda^\Phi,\peq_\lambda^\Phi, V_\lambda^\Phi\>$ where 
\begin{enumerate}
\item $W_\lambda^\Phi =\mathrm{type}^A_{\lambda}(\Phi)$
\item For each $p\in \At\cap \Gamma$, $\Gamma\in V_\lambda^\Phi(p)$ iff $p\in \Gamma$, 
\item $E_\lambda^\Phi(\Gamma)=\cbra \mathcal N^{\Phi}_{\lambda}(\alpha)\ |\ \nc\alpha\in\Gamma\cket,$
\item $\Gamma \mathrel B_\lambda^\Phi\Delta$ if $\Gamma^{\nB}\subseteq \Delta$,
\item $\Gamma \peq_\lambda^\Phi\Delta$ if $\Gamma^{[\peq]}\subseteq \Delta^{[\peq]}$.
\end{enumerate}
\end{definition}

Note that our `canonical' extended evidence models  are not unique, as they  depend on the type $\Phi$ we wish to satisfy and the logic $\lambda$ we wish to work in; in this sense they could be considered `semicanonical'.   Considering different extended evidence models in this setting is unavoidable, given that our language includes a universal modality and it is impossible to satisfy all types on a single model.

Let us begin by showing that canonical models are models in the sense of Definition \ref{ev-model}:

\begin{lemma}\label{nonempt}
Given a consistent $\lambda$-type $\Phi$ with a  $\lambda$-canonical extended evidence model  $\mo M_\lambda(\Phi),$
\begin{enumerate}
\item For each $\Gamma\in W_\lambda^\Phi$, $\Gamma\mathrel E_\lambda^\Phi W_\lambda^\Phi$;\label{nonempt1}
\item If $\nc\alpha\in\Phi$,  $\mathcal N^{\Phi}_{\lambda}(\alpha)$ is non-empty; \label{nonempt2}

\item The relation $\peq^\Phi_\lambda$ is a preorder, and whenever $\Gamma\peq_\lambda^\Phi \Delta$ and $\Gamma\in X\in E(\Theta)$ it follows that $\Delta\in X$.
\item If $\Gamma\peq^\Phi_\lambda\Delta$ and $\Theta\mathrel B^\Phi_\lambda \Gamma$ then $\Theta\mathrel B^\Phi_\lambda \Delta$.
\end{enumerate}
\end{lemma}

\begin{proof}
Suppose that $\Phi$ is a $\lambda$-type.

\paragraph{1} The formulas $\nc\top,\top$ are derivable, so it follows that $\top \in \Delta$ for all $\Delta\in W^\Phi_\lambda$; thus $\mathcal N^{\Phi}_{\lambda}(\top)=W^\Phi_\lambda$. Similarly, $\nc\top\in \Gamma$, so that $\Gamma \mathrel E_\lambda^\Phi W_\lambda^\Phi$, as claimed.

\paragraph{2} Put $\delta=\top$ in Lemma \ref{cons}.\ref{cons1} and observe that any $\lambda$-type $\Delta$ extending $\{\alpha\wedge\nl\alpha,\top\}\cup \Gamma^\lambda$ belongs to $\mathcal N^{\Phi}_{\lambda}(\alpha)$.

\paragraph{3} The transitivity of $\peq^\Phi_\lambda$ follows by definition. For the second condition, let $X=\mathcal N^{\Phi}_{\lambda}(\alpha)$ be any neighborhood of $\Theta$. Since $\Gamma\in X$ it follows that $\alpha,[\peq]\alpha\in \Gamma$ so that $[\peq]\alpha\in \Delta$. But then, $\alpha\in\Delta$ as well by the truth axiom, so $\Delta\in X$.

\paragraph{4} If $\Gamma\peq^\Phi_\lambda\Delta$ and $\Theta\mathrel B^\Phi_\lambda \Gamma$, we need to check that, given $\nB\phi\in \Theta$, we have $\phi\in \Delta$. But by {\bf $B$-monotonicity} we have $\nB\nl \varphi\in \Theta,$ so that $\nl\varphi\in\Gamma$ and thus $\nl\varphi\in \Delta$, which by the truth axiom implies $\varphi\in\Delta$.
\end{proof}

The next lemmas are the key ingredients in the proof of the Truth Lemma.   
 
\begin{lemma}\label{boxdia} Let $\lambda$ be an evidence logic over the signature $EBA\mathord\peq$, and $\Phi$ be a $\lambda$-type.
\begin{enumerate}
\item If $\nc\alpha,\ps\beta\in\Gamma\in W^\Phi_\lambda$, there is $\Delta\in\mathcal N^{\Phi}_{\lambda}(\alpha)$ with $\beta\in\Delta$.\label{boxdia1}

\item If $\langle X\rangle\alpha\in \Gamma$ with $X\in \{B,\peq\}$ there is $\Delta$ such that $\Gamma \mathrel X_\lambda^\Phi\Delta$ and $\alpha\in\Delta$.\label{boxdia2}

\item If $\lambda\in\{\logic_{\flat},\logic_{\flat u}\}$ and $\alpha\in \Gamma^\nc$,  then $\mathcal N^{\Phi}_{\lambda}(\alpha)\cap \mathrm{type}^B_{\lambda}(\Gamma)$ is non-empty.\label{boxdia2}

\end{enumerate}
\end{lemma}

 \begin{proof}
Let $\Phi$ be a $\lambda$-type.

\paragraph{1} Since $\Gamma$ is $\lambda$-consistent  and $\nc\alpha,\ps\beta\in\Gamma$, we have, by Lemma \ref{cons}.1, that
$$\{\alpha\wedge\nl\alpha,\beta\}\cup \Gamma^\lambda$$
is $\lambda$-consistent.  This can be extended to a $\lambda$-type $\Delta$.   Then, by definition we have $\Delta\in\mathcal N^{\Phi}_{\lambda}(\alpha)$ and $\beta\in \Delta$, as desired.

\paragraph{2} For $X=B$, by Lemma \ref{cons}.2, since $\Phi$ is $\lambda$-consistent, we have that
$$\{\beta\}\cup \Gamma^{\nB}\cup \Gamma^\lambda$$ 
is $\lambda$-consistent as well.  Thus, it can be extended to a $\lambda$-type $\Delta$. Evidently, $\Gamma \mathrel B^\Phi_\lambda\Delta$ and $\beta\in\Delta$.

The case $X=\mathord\peq$ is analogous but uses Lemma \ref{cons}.3. 

\paragraph{3} By Lemma \ref{cons}.4, if $\Gamma$ is $\lambda$-consistent and $\lambda$ is flat, then  $\{\alpha\}\cup \Gamma^\nB\cup \Gamma^\lambda$ is $\lambda$-consistent as well, and hence can be extended to a $\Phi$-type $\Delta$. Obviously, 
\[\Delta\in \mathcal N^{\Phi}_{\lambda}(\alpha)\cap \mathrm{type}^B_{\lambda}(\Gamma).\]
\end{proof}
 
Now let us check that canonical models fall within the appropriate class of models for each logic.

\begin{lemma}\label{isunif}
Let $\lambda$ be an evidence logic over $EBA\mathord\peq$ and $\Phi$ be $\lambda$-consistent. Then,
\begin{enumerate}
\item If $\lambda$ is uniform, then so is $\mo M_\lambda(\Phi)$.
\item If $\lambda$ is flat, then so is $\mo M_\lambda(\Phi)$.
\item If $\lambda=\logic^\sigma_{\flat u}$, then $\mo M_\lambda(\Phi)$ is concise.
\end{enumerate}

\end{lemma}

\newpage
\begin{proof}\
\paragraph{1} In the uniform case, we have that both the accessible states and the neighborhoods of $\Psi$ depend only on $\Psi^\lambda$, which is constant amongst all of $W^\Phi_\lambda$. It remains to check that if $\Gamma\in B_\lambda^\Phi$ then $\Gamma$ is $\peq^\Phi_\lambda$-maximal.

So suppose that $\Gamma \peq^\Phi_\lambda\Delta$ and $[\peq]\alpha\in\Delta$. It follows that $\pl\nl\alpha\in\Gamma$. If moreover $\pl\neg\alpha\in\Gamma$, then $\pl\nl\alpha\wedge\pl\neg\alpha\in\Gamma$ and hence  $\pB(\pl\nl\alpha\wedge\pl\neg\alpha)\in\Phi$. But then, by the {\bf maximality} axiom and transitivity of $\peq$, $\pB(\nl\alpha\wedge\pl\neg\alpha)\in\Phi$, which cannot be, since $\Phi$ is consistent.

We then conclude that $\nl\alpha\in\Gamma$ and, since $\alpha$ was arbitary, $\Delta\peq^\Phi_\lambda\Gamma$, as required.

\paragraph{2} This amounts to showing that if $\Gamma\mathrel E_\lambda^\Phi X$, there is $\Delta\in X$ such that $\Gamma\mathrel B_\lambda^\Phi\Delta$. But this is immediate from the definition of the accessibility relation and neighborhoods on canonical extended evidence models and Lemma \ref{boxdia}.\ref{boxdia2}.

\paragraph{3} As we have already seen that $\mo M_\lambda(\Phi)$ is both flat and uniform, it remains to check that if $\Gamma$ is $\peq^\Phi_\lambda$-maximal then $\Gamma\in B^\Phi_\lambda$.

Whenever $\nB\phi\in\Gamma$ it follows that $\pl\nl\phi\in\Gamma$ by {\bf conciseness}. If $\Gamma$ is $\peq^\Phi_\lambda$-maximal then it also follows that $\phi\in \Gamma$; otherwise $\nl\phi\not\in \Gamma$, while by Lemma \ref{boxdia}.\ref{boxdia2} there would be $\Delta$ with $\Gamma\peq^\Phi_\lambda \Delta$ and $\nl\phi\in\Delta$, so that $\Gamma\prec^\Phi_\lambda \Delta$. But this would contradict the maximality of $\Gamma$. Thus $\Gamma \mathrel B^\Phi_\lambda\Gamma$, and since $B^\Phi_\lambda$ is constant this means that $\Gamma\in B^\Phi_\lambda$.
\end{proof}

 We are now ready to prove our own version of the `Truth Lemma':

\begin{proposition}[Truth Lemma]\label{truth}  Given a $\lambda$-type $\Phi$ we have that, for every formula $\psi$ and every $\Gamma\in \mathrm{type}^A_{\lambda}(\Phi)$,
\[\psi\in \Gamma\Leftrightarrow \Gamma\in \truth{\psi}{\mo{M}_\lambda(\Phi)}.\]
\end{proposition}

\begin{proof} The proof is by induction on the structure of $\psi$.   The only interesting cases are the modalities $\nc$, $\nB$ and $\nl$ (and their duals).  

Suppose first that $\nc\psi\in \Gamma$, and let $X=\mathcal N^{\Phi}_{\lambda}(\psi)$. By definition, we have  $\psi\in \Theta$ for every $\Theta\in X$, and by the induction hypothesis this implies that for each $\Theta\in X$, $\Theta\in \truth{\psi}{\mo{M}_\lambda(\Phi)}$.   Thus,   $X$ is a neighborhood of $\Gamma$ ($X\in E^\Phi_\lambda(\Gamma)$)  with $X\subseteq \truth{\psi}{\mo{M}_\lambda(\Phi)}$, so $\Gamma\in\truth{\nc\psi}{\mo{M}_\lambda(\Phi)}$, as desired.

Now, assume that $\ps\psi\in \Gamma$. Let $X$ be any neighborhood of $\Gamma$. Then $X= \mathcal N^{\Phi}_{\lambda}(\alpha)$ for some $\alpha$. Hence, $\nc\alpha\in\Gamma$ and $\ps\psi\in\Gamma$.  By Lemma \ref{boxdia}.1, there is $\Theta\in \mathcal N^{\Phi}_{\lambda}(\alpha)$ with $\psi\in\Theta$. By the induction hypothesis, $\Theta\in \truth{\psi}{\mo{M}_\lambda(\Phi)}$, and,  since $X$ was arbitrary, we conclude that $\Gamma\in\truth{\ps\psi}{\mo{M}_\lambda(\Phi)}$.

Let us now consider $\nB\psi\in\Gamma$.  Suppose that $\Gamma \mathrel B_\lambda^\Phi\Delta$.   Then, by definition $\psi\in\Delta$, which by the induction hypothesis implies $\Delta\in \truth{\psi}{\mo{M}_\lambda(\Phi)}$.   Thus, $\Gamma\in \truth{\nB\psi}{\mo{M}_\lambda(\Phi)}$

Assume next that  $\pB\psi\in \Gamma$.  By  Lemma \ref{boxdia}.3, $\psi\in\Theta$ for some $\Theta\in \mathrm{type}^B_\lambda(\Gamma)$. But then $\Gamma \mathrel B_\lambda^\Phi\Theta$, and thus $\Gamma\in\truth{\pB\psi}{\mo{M}_\lambda(\Phi)}$.

The cases for $\nl\psi,\pl\psi$ are similar and we skip them.
 \end{proof}

We recall that for $\mathcal M^\triangledown_{\flat u}$ to be concise, we require that $\mathcal M$ be concise {\em and} finite. Thus we need to produce finite models; we shall do so by a filtration argument, which requires only a slight modification from the standard construction.

\begin{definition}
Given a $\lambda$-type $\Phi$ and a formula $\varphi$, we define $\mathcal M_\lambda(\Phi)/\varphi=\langle W^{\Phi/\varphi}_\lambda,E^{\Phi/\varphi}_\lambda, B^{\Phi/\varphi}_\lambda,\peq^{\Phi/\varphi}_\lambda, V^{\Phi/\varphi}_\lambda\rangle$ as follows:
\begin{enumerate}
\item Write $\Gamma\equiv_\varphi\Delta$ if
\begin{enumerate}
\item for every subformula $\psi$ of $\varphi$, $\psi\in \Gamma$ if and only if $\psi\in \Delta$, and
\item $\Gamma$ is $\peq^\Phi_\lambda$-maximal if and only if $\Delta$ is.
\end{enumerate} 
Then, let $W^{\Phi/\varphi}_\lambda$ be the set of equivalence classes of ${\rm type}^A(\Phi)$ modulo $\equiv_\varphi$. We denote the equivalence class of $\Gamma$ by $[\Gamma]_\varphi.$
\item For $X=B,\peq$, $[\Gamma]_\varphi\mathrel X^{\Phi/\varphi}_\lambda[\Delta]_\varphi$ if and only if there are $\Gamma'\equiv_\varphi\Gamma$ and $\Delta'\equiv_\varphi\Delta$ with $\Gamma'\mathrel X^\Phi_\lambda\Delta'$.

\item $[\Gamma]_\varphi\mathrel E^{\Phi/\varphi}_\lambda X$ if there is $\Gamma'\equiv_\varphi\Gamma$ and $Y\in E^\Phi_\lambda(\Gamma)$ such that $X=[Y]_\phi=\{[\Delta]_\varphi:\Delta\in Y\}$.

\item If $p$ is a subformula of $\varphi$ then $V^{\Phi/\varphi}_\lambda(p)=\{[\Gamma]_\varphi:p\in \Gamma\}$; otherwise, $V^{\Phi/\varphi}_\lambda(p)=\varnothing$.
\end{enumerate}
\end{definition}

The following is standard and we present it withouth proof:

\begin{lemma}\label{filtertruth}
If $\psi$ is a subformula of $\varphi$, then $\Gamma\in \lb\psi\rb_{\mathcal M_\lambda(\Phi)}$ if and only if $[\Gamma]_\varphi\in\lb\psi\rb_{\mathcal M_\lambda(\Phi)/\varphi}$.
\end{lemma}

To check that filtered models preserve conciseness, we must make sure that they preserve maximal elements:

\begin{lemma}\label{maxfinlemm}
Given a $\lambda$-type $\Phi$, a formula $\varphi$ and $\Gamma\in W^\Phi_\lambda$, $[\Gamma]_\varphi$ is $\peq^{\Phi/\varphi}_\lambda$-maximal if and only if $\Gamma$ is $\peq^{\Phi}_\lambda$-maximal.
\end{lemma}

\begin{proof}
If $\Gamma$ is maximal and $[\Gamma]_\varphi \peq^{\Phi/\varphi}_\lambda[\Delta]_\lambda$, there are $\Gamma'\equiv_\varphi\Gamma$ and $\Delta'\equiv_\varphi\Delta$ such that $\Gamma'\peq^\Phi_\lambda\Delta'$; but then, $\Gamma'$ is maximal by definition so that $\Delta'\peq \Gamma'$ and thus $[\Delta]_\varphi \peq^{\Phi/\varphi}_\lambda[\Gamma]_\varphi$, establishing the maximality of $[\Gamma]_\varphi$.

For the other direction, it suffices to show that, given $\Gamma\in W^\Phi_\lambda$, there is $\Delta\seq\Gamma$ such that $\Delta$ is maximal; for then, if $\Gamma$ is not maximal, then no $\Gamma'\equiv_\phi\Gamma$ is maximal, whereas every $\Delta'\equiv_\phi\Delta$ is maximal, so $\Gamma'\not\seq\Delta'$ and $[\Delta]_\varphi\succ^{\Phi/\varphi}_\lambda [\Gamma]_\varphi$.

But if $X\subseteq W^{\Phi/\varphi}_\lambda$ is totally ordered, it is not hard to check that $X$ has an upper bound (pick any $\lambda$-type extending $\nl\left(\bigcup X\right)^\nl$). Thus by Zorn's lemma, the set of upper bounds of $\Gamma$ has a maximal element, as needed.
\end{proof}

With this we may prove the following:

\begin{lemma}\label{hasprop}
Given a $\flat u$-type $\Phi$ and a formula $\varphi$, $\mathcal M_{\flat u}(\Phi)/\varphi$ is finite and concise.
\end{lemma}

\begin{proof}
The finiteness is obvious because there can only be finitely many $\equiv_\varphi$-equivalence classes. It is not too hard to see that $\mathcal M_{\flat u}(\Phi)/\varphi$ is flat, for if $X=[Y]_\varphi$ is any neighborhood of $[\Gamma]_\varphi$, by flatness we have that $\Gamma \mathrel B^\Phi_{\flat u}\Delta$ for some $\Delta\in Y$ and thus $[\Gamma]_\varphi \mathrel B^{\Phi/\varphi}_{\flat u}[\Delta]_\varphi\in X$. It is also easy to see that $B^{\Phi/\varphi}_{\flat u}$, $E^{\Phi/\varphi}_{\flat u}$ are constants given that $B^{\Phi}_{\flat u}$, $E^{\Phi}_{\flat u}$ are.

It remains to check the maximality clause. But this is immediate from Lemma \ref{maxfinlemm}, for indeed $[\Gamma]_\varphi$ is maximal if and only if $\Gamma$ is, which in turn is equivalent to $\Gamma\in B^\Phi_{\flat u}$ and thus to $[\Gamma]_\varphi\in B^{\Phi/\varphi}_{\flat u}$.
\end{proof}

\subsection{Completeness theorem}
With the work developed in previous sections, we are ready to state and prove our main completeness results.

\begin{theorem} Each evidence logic $\lambda$ over the full signature $EBA\mathord\peq$ is sound and strongly complete for the class of $\lambda$-models. Moreover,
\begin{enumerate}
\item $\logic, \logic_\flat,\logic_{\flat u}$ are sound and strongly complete for their class of evidence models,
\item $\logic_{\flat u}$ is sound and weakly complete for its class of finite evidence models.
\end{enumerate}
\end{theorem}

\begin{proof}
Let $\Phi$ be $\lambda$-consistent and let $\Theta$ be a $\lambda$-type extending $\Phi$. Then, let $\mo{M}_\lambda(\Theta)$ be its canonical $\lambda$-structure. By Lemma \ref{isunif}, $\mo{M}_\lambda(\Phi)$ is a $\lambda$-model, and by Proposition \ref{truth}, given $\phi\in\Theta$, $\Theta\in\truth{\phi}{\mo{M}_\lambda(\Theta)}$, so that $\Phi$ is satisfied by $\mo{M}_\lambda(\Theta)$. By Theorem \ref{reptheo}, $\mo {M}_\lambda (\Theta)_\lambda^\triangledown$ is a $\lambda$-evidence model also satisfying $\Theta$ if $\lambda\not=\logic_{\flat u}$.

The completeness of $\logic_{\flat u}$ for its evidence models requires a filtration. Let $\varphi$ be $\flat u$-consistent and $\Phi$ be a $\flat u$-type extending $\varphi$. Then, consider $\mathcal M_{\flat u}(\Phi)/\varphi$. By Lemma \ref{hasprop}, $\mathcal M_{\flat u}(\Phi)/\varphi$ is finite, concise and by Lemma \ref{filtertruth}, satisfies $\varphi$. It then follows from Theorem \ref{reptheo} that $(\mathcal M_{\flat u}(\Phi)/\varphi)^\triangledown_{\flat u}$ is a finite, concise evidence model which satisfies $\varphi$.
\end{proof}

We know that $\logic$ and $\logic_u$ do not have the finite evidence model property, but the question of whether they have finite extended models remains open. Other important questions left open are: does $\logic_\flat$ have the finite evidence model property? Is $\logic_{\flat u}$ strongly complete for its class of evidence models? We conjecture affirmative answers to these questions, but leave them for future inquiry.

\section{Dynamic extensions}\label{lang-ext}

The logical systems studied in the previous sections are interesting in their own right as they combine features of both normal and non-normal modal logics in a non-trivial way. An additional appealing feature of the systems studied here, following their original introduction in \cite{vBP-evidence}, is our fresh interpretation of neighborhood structures.   Viewing neighborhoods as bodies of evidence suggested a number of new and interesting modalities beyond the usual repertoire studied in modal neighborhood logics. We will now show briefly how this process of language elicitation can be taken further when we take a further step, toward a dynamic perspective.  

\medskip
Evidence is not a static substance that we have once and for all. It is continually affected by new incoming information, and also by processes of internal re-evaluation.    Taking a cue  from recent dynamic logics of knowledge update \cite{vDvdHK,vanbenthem-newbook} and belief revision \cite{vb-logbelrev, baltag-smets}, van Benthem and Pacuit \cite{vBP-evidence} study a number of dynamic extensions of evidence models.  They develop a  rich repertoire of dynamic logics of   ``evidence management".  We will follow one such thread here, showing the new logical issues that arise. 

 The simplest type of information change is   receiving information form an infallible source.     This operation is called {\em public announcement}   \cite{Pl:89,G:98} and it transforms a model by removing all states where the announced formula is false.   From the perspective of our evidence models, such a basic action in possible worlds models is actually a compound of various transformations. A public announcement of $\phi$ can be naturally ``deconstructed'' into a complex combination of three distinct operations: 

\begin{enumerate}
\item {\bf Evidence addition}: the agent accepts that $\phi$ is an ``admissible" piece of evidence (perhaps on par with the other available evidence).  
 
\item  {\bf Evidence removal}: the agent removes any  evidence for $\neg \phi$.  
 
\item  {\bf Evidence modification}: the agent incorporates $\phi$ into each piece of evidence gathered so far, making $\phi$ the most important piece of evidence. 

\end{enumerate}
Our richer evidence models allows us to study each of these operations individually (cf. \cite{vBP-evidence}).   In this section, we focus on just one of these: accepting an input from a trusted source, leaving other dynamic operations for further study. Let $\M$ be an evidence model  and $\phi$ a new piece of evidence which the agent decides 
 to {\em accept}.  Here ``acceptance" does not necessarily mean that the agent believes that $\phi$ is true, but rather that she agrees that $\phi$ should be considered when ``weighing" her evidence.  The formal definition of this action is straightforward: 
 
\begin{definition}[Evidence Addition]  Let $\M=\langle W, E, V\rangle$ be an evidence model, and $\phi$ a formula in $\L$.   The model $\M^{+\phi}=\langle W^{+\phi},E^{+\phi},V^{+\phi}\rangle$ has $W^{+\phi}=W$, $V^{+\phi}=V$ and for all $w\in W$, 
\begin{center} $E^{+\phi}(w)=E(w)\cup\{\truth{\phi}{\M}\}.$   \end{center}
\end{definition}
This operation can be described explicitly in our modal language with a dynamic modality  $[+\phi]\psi$ which is intended to mean ``$\psi$ is true after $\phi$ is accepted as an admissible piece of evidence".   The truth condition for this formula is straightforward:  

\begin{center} (EA)\hspace{.2in} $\M,w\models [+\phi]\psi$ iff    $\M,w\models E\phi$ implies $\M^{+\phi},w\models\psi$.  \end{center}
Here, since evidence sets cannot be empty, the precondition for evidence addition is that $\phi$ is true at some state. Compare this with the better-known precondition for public announcement which requires the  accepted formula to be true.    The logical analysis of even this simple operation is surprisingly subtle, suggesting  a number of interesting extensions to the modal languages for evidence studied in this paper.  

\medskip
Let us start with laws of reasoning for the  dynamic modalities, extending the logics that we had so far. For languages only containing the evidence and universal modality, we have the following new kind of validities, that may be viewed as ``recursion laws'' for evidence and knowledge modalities after an act of evidence addition: 

\begin{center}
\begin{tikzpicture}

\node[rectangle,  , rounded corners] {\begin{minipage}{4in}
\vspace{.1in}\begin{tabular}{lll}  
 $[+\phi]p$ & $\leftrightarrow$  & $(\langle A\rangle\phi\rightarrow p)$\hspace{.2in} ($p\in\At$)\\[0.2cm]
   $[+\phi](\psi\wedge\chi)$ & $ \leftrightarrow$ & $ ([+\phi]\psi\wedge [+\phi]\chi)$\\[0.2cm]
  $[+\phi]\neg\psi$ & $\leftrightarrow $ & $(\langle A\rangle\phi\rightarrow \neg[+\phi]\psi)$\\[0.2cm]
   $[+\phi] [E]\psi$ & $ \leftrightarrow  $ & $(E\phi\rightarrow ([E][+\phi]\psi\vee [A](\phi\rightarrow [+\phi]\psi)))$\\[0.2cm]
    $[+\phi] A\psi $ & $\leftrightarrow$& $ (\langle A\rangle \phi\rightarrow [A][+\phi]\psi)  $\\[0.2cm]

\end{tabular}

\end{minipage}};
 \end{tikzpicture}
\end{center}
But things can quickly get more complicated. Finding a similar recursion law for the belief operator of this paper requires an extension to our basic language.    This is the operator $B^\phi\psi$ of {\em conditional belief}: ``the agent believes $\psi$ conditional on $\phi$", which is well-known, but now with a new twist in our setting. Some of the agent's current evidence may be inconsistent with $\phi$ (i.e., disjoint with $\truth{\phi}{\M}$).  If one is restricting attention to situations where $\phi$ is true, then such inconsistent evidence must be ``ignored''.  Here is how we can deal with this in our models:

\begin{definition}[Relativized scenarios.]  Suppose that $X\subseteq W$.  Given a collection $\X$ of subsets of $W$ (i.e., $\X\subseteq \pow(W)$), the relativization of $\X$ to $X$ is the set $\X^X=\{Y\cap X\ |\ Y\in\X\}$.   We say that a collection $\X$ of subsets of $W$ has the  {\bf finite intersection property relative to $X$} ($X$-f.i.p.) if, for each $\{X_1,\ldots,X_n\}\subseteq \X^X$, $\bigcap_{1\le i\le n} X_i\ne\emptyset$.   We say that $\X$ is a $w$-scenario relative to $X$ when $\X\subseteq E(w)$ and   $\X$ has the {\bf maximal $X$-f.i.p.}. \end{definition}
To simplify notation, when $X$ is the truth set of formula $\phi$, we write  ``$\X^\phi$" for ``$\X^{\truth{\phi}{\M}}$" and ``$\X$ is a $\phi$-scenario (for $w$)" for ``$\X$ is a $w$-scenario relative to $\truth{\phi}{\M}$".    Now we define a natural notion of conditional belief: 
 
  \begin{itemize}  
\item $\M,w\models \nB^\phi\psi$ iff for each   $\phi$-scenario $\X\subseteq E(w)$, for each $v\in\bigcap\X^\phi$, $\M,v\models\psi$
\end{itemize}

Conditional belief is an entirely natural notion in probability and belief revision theory. Still, this notion suggests a logical investigation beyond what we have provided in this paper. First of all, strikingly, $\nB\phi\rightarrow \nB^{\psi}\phi$ is not valid.  One can compare this to the failure of monotonicity for antecedents in conditional logic.  In our more general setting which allows inconsistencies among accepted evidence, even the following variant is not valid: $\nB\phi\rightarrow (\nB^{\psi}\phi\vee \nB^{\neg\psi}\phi$).  To see this, consider an evidence model with $E(w)=\{X_1, Y_1, X_2, Y_2\}$  where the sets are defined as follows: 
 
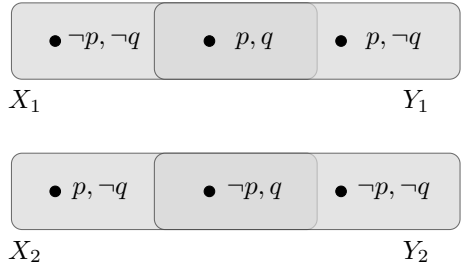
\begin{center}

\begin{tikzpicture}

\node[rectangle,draw=black!80,fill=gray!30,opacity=0.7,minimum width=1.6in,minimum height=0.4in,rounded corners] at (0.65,0) {$ $}; 
\node[rectangle,draw=black!80,fill=gray!30,opacity=0.7,minimum width=1.6in,minimum height=0.4in,rounded corners] at (2.55,0) {$ $};

\node at (-1.2,-0.8) {$X_1$}; 
\node at (4,-0.8) {$Y_1$}; 

\draw[fill]    (-0.8,0) circle (2pt); 
\node at (-0.15,0) {$\neg p,\neg q$}; 

 \draw[fill]    (1.25,0) circle (2pt); 
 \node at (1.85,0) {$p,q$}; 

\draw[fill]    (3,0) circle (2pt); 
\node at (3.7,0) {$p,\neg q$};

\node[rectangle,draw=black!80,fill=gray!30,opacity=0.7,minimum width=1.6in,minimum height=0.4in,rounded corners] at (0.65,-2) {$ $}; 
\node[rectangle,draw=black!80,fill=gray!30,opacity=0.7,minimum width=1.6in,minimum height=0.4in,rounded corners] at (2.55,-2) {$ $}; 

\node at (-1.2,-2.8) {$X_2$}; 
\node at (4,-2.8) {$Y_2$}; 

\draw[fill]    (-0.8,-2) circle (2pt); 
\node at (-0.2,-2) {$p, \neg q$}; 

 \draw[fill]    (1.25,-2) circle (2pt); 
 \node at (1.85,-2) {$\neg p, q$}; 

\draw[fill]    (3,-2) circle (2pt); 
\node at (3.7,-2) {$\neg p, \neg q$}; 
 \end{tikzpicture}

\end{center}
Then, $\M,w\models \nB q$; however,  $\M,w\not\models \nB^p q\vee \nB^{\neg p} q$.   This counter-example is interesting, as our putative implication is indeed valid on {\em connected} plausibility models for conditional belief (cf. \cite{Board} for a complete modal logic of conditional belief on such models).  
\medskip
Now in principle, our earlier axiomatization captures conditional belief as well, since there is a definition for it in terms of safe belief that arises from relativizing the earlier definition of belief in terms of safe belief and the universal modality (cf. \cite{baltag-smets-loft06,boutilier,vanbenthem-newbook}). But things get more subtle when we consider the dynamics of evidence addition once more. Clearly, to keep the total system in harmony, we must provide a recursion law, not just for pure beliefs, but also for conditional beliefs. And then the earlier static base language requires extension once more:

\medskip
 The notion of conditional belief alone does not yet allow us to state a recursion axiom.   To see this, note that $B^{\phi}\psi$ may be true at a state $w$ without having any $w$-scenarios that imply $\phi$ (i.e., there is no $w$-scenario $\X$ such that $\bigcap\X\subseteq \truth{\phi}{\M}$).    A  more general form of conditioning appropriate to this setting is $B^{\phi,\alpha} \psi $ where   ``the agent believes that $\psi$, after having settled on $\alpha$ and conditional on $\phi$".   Formally, 

\begin{itemize}
\item $\M,w\models \nB^{\phi,\psi}\chi$ iff for all maximally $\phi$-compatible sets $\X\subseteq E(w)$, if $\bigcap\X\cap\truth{\phi}{\M}\subseteq \truth{\psi}{\M}$, then $\bigcap\X\cap\truth{\phi}{\M}\subseteq\truth{\chi}{\M}$.   
\end{itemize}
Note that   $\nB^{+\phi}$ can be defined as $B^{\phi,\top}$.  We also write $\nB^{-\phi}$ for $\nB^{\top,\neg\phi}$.   With this operator added, we can   indeed state  a recursion law for our operators of absolute and conditional belief (the proofs that these axioms are valid can be found in \cite{vBP-evidence}):   

\medskip

\begin{center}
 \begin{tikzpicture}[scale=0.7]
\node[rectangle,    rounded corners] { 
\scalebox{0.9}{ \begin{tabular}{rll}  
  $[+\phi] \nB\psi $ & $\leftrightarrow$ & $  (\langle A\rangle\phi\rightarrow (\nB^{+\phi}[+\phi]\psi \wedge \nB^{-\phi} [+\phi]\psi))$ \\[0.2cm]
   
  $[+\phi] \nB^{\psi,\alpha}\chi $ & $\leftrightarrow$ & $ (\langle A\rangle\phi\rightarrow (\nB^{\phi\wedge[+\phi]\psi,[+\phi]\alpha}[+\phi] \chi\wedge \nB^{[+\phi]\psi,\neg\phi\wedge[+\phi]\alpha}[+\phi] \chi)) $\\[0.2cm]
 
\end{tabular}}};
 \end{tikzpicture}
\end{center}

\medskip

Finally, recursion laws for evidence addition and our plausibility modalities also raise interesting new issues.  Recall that a condition on  our models (Definition \ref{ev-model}) is that the evidence sets are upwards closed under the plausibility relation.     This means that   the precondition for adding a new piece of evidence to an arbitrary model is that it is upwards closed under the plausibility ordering.  Alternatively, we can introduce a new evidence addition operation $E^{\oplus X}(w)=E(w)\cup\{X'\}$ where $X'=\{u\ |\ \text{ there is a $v\in X$  such that $v\peq u$}\}$.  We leave a full analysis of this new operation for future work.

\medskip

What we have shown so far is how a logical analysis of evidence dynamics that changes given neighborhood models for agents uncovers the need for new modalities in our base language of evidence models. Clearly, there are many new questions of axiomatization resulting from this, but other kinds of issues emerge as well. We conclude with just a simple illustration.

\medskip

There is a natural issue of ``harmony'' between two perspectives on evidence dynamics. As we have seen, on evidence models $\M=\langle W, E, V\rangle$, we can   {\em derive} the plausibility ordering directly from the evidence set (see Definition \ref{derived-plaus}). Viewed in this way, an addition to   the agent's evidence also changes the agent's plausibility ordering.  But this means that we can also look at evidence change in terms of transformations of models of a more familiar relational kind, described in Section \ref{derived}: namely, plausibility orderings of epistemic ranges. Belief change on models of the latter kind has been studied extensively: cf. \cite{baltag-smets,vanbenthem-newbook,fenrong-book}, and this suggests a parallel line of inquiry  finding dynamic operations on plausibility orderings matching changes in the agent's evidence.    

For evidence addition, the corresponding technical question is which operation makes the following diagram commute:  

\begin{center}\begin{tikzpicture}

\node (E1) at (0,1.5) {$\E$};
\node (preceq1) at (0,0) {$\peq_E$}; 
\node (E2) at (2,1.5) {$\E^+$};
\node (preceq2) at (2,0) {$\peq^{+}$}; 

\path[->,draw,densely dashed,rounded corners] (E1) to (preceq1); 
\path[->,draw,densely dashed,rounded corners] (E2) to (preceq2);

\path[->,double,draw] (E1) to node[above] {\footnotesize $+X$} (E2);
\path[->,double,draw] (preceq1) to node[below] {\footnotesize $??$} (preceq2);
\end{tikzpicture}
\end{center}
 And the answer is not hard to find. Suppose that $\peq^{+}=\peq - \{(w,v)\ |\ w\in X\text{ and } v\not\in X\}$.  Then, 
$$\peq^+=\peq_{\E^+}$$
This simple observation is in line with the study of ``two-level models'' for preference dynamics in \cite{fenrong-book}, to which we refer for further relevant results. The extent of these harmony phenomena remains to be investigated, where we expect our earlier distinction between general and intended models, suitably generalized to the dynamic setting, to be important once more.

\bigskip

Our  discussion of the dynamics of evidence change shows one further way in which neighborhood structures are a good vehicle for exploring finer epistemic and doxastic distinctions than those found in standard relational models for modal logic.  At the same time, this richness means that new techniques may be needed in the logical analysis of this richer form of neighborhood semantics. Model-theoretically, we need stronger notions of bisimulation and {$p$}-morphism matching the more expressive languages that emerged in the formulation of dynamic recursion laws, while proof-theoretically, we need to lift our earlier completeness technique to this richer setting.

\section{Concluding remarks}\label{conclusion}

Starting from earlier work on evidence logics in neighborhood models, this paper has offered two main new contributions. First, our completeness theorems contribute  to the general study of basic evidence logic and its dynamics. Second, in doing so, we develop a new perspective on neighborhood models that suggests extensions of the usual systems, even without any specific evidence interest in mind.  For the modal logician, the pleasant surprise of our analysis is that there is a lot of well-motivated new modal structure to be explored on models that are often considered totally explored.

Clearly, many open problems remain. These start at the base level of the motivations for our framework, touched upon lightly earlier on. For instance, our framework still needs to be related to other modal logics of evidence \cite{shafer,halpern-pucella} and argumentation \cite{davide}. Closer to what we have presented, however, here are a few more specific avenues for future research:

\begin{itemize}

\item {\em Further interpretations}:  By imposing additional constraints on the evidence relations (i.e., the neighborhood functions), we get evidence models that are topological spaces.  Can we also give a spatial interpretation to our belief operator and the new matching modalities discussed in Section \ref{lang-ext}? This would suggest new, richer  modal logics for reasoning about topological spaces (cf. \cite{spatial-logic}).

\item {\em Computational complexity}:  Neighborhood logics are often NP-complete, while basic modal logics on relational models are often Pspace-complete. What about mixtures of the two?  

\item {\em Extended model theory}:  Combining existing notions of bisimulation for relational models and neighborhood models takes us only so far. What new notions of bisimulation directly on evidence models will match our extended modal languages?  And in this setting, can the key representation method of this paper (Theorem \ref{reptheo}) be extended to deal with such richer languages?

\item {\em Extended proof theory}:  How can we axiomatize the richer modal logic of neighborhood models that arises in our dynamic analysis? Will the techniques of our completeness proofs carry over, or do we need a new style of analysis of canonical models?

\end{itemize}

Once again, we end with Sergei Artemov's work on justification logics. As we have said in the Introduction, our evidence models provide a coarser level of analysis somewhere above the syntactic detail of proof terms.  In conclusion, we briefly highlight two further comparisons.      In an evidence model,  the agent has evidence for $\phi$ provided there is a proposition (i.e., a set of states)  that the agent has identified as evidence which  logically implies $\phi$.  An important feature of justification logic is that a {\em proof term} not only records the fact that certain beliefs are grounded in evidence, but also the  proof of why that evidence justifies a particular belief.   A second point of comparison is that our  notion of evidence   is weaker than what is found in the justification logic literature.   Having a justification for $\phi$ means, in particular, that the agent believes that $\phi$ is true.   Contrast this with our setting in which an agent only believes what is implied by {\em all} of her evidence.    Clearly,  this only scratches the surface and there is much more to say about the connection with justification logic, but this will be left for future work.

\bibliographystyle{plain}

\end{document}